\documentclass{amsart}
\usepackage[utf8]{inputenc}
\usepackage{amssymb}
\usepackage{amsmath}
\usepackage{amsthm}
\usepackage{hyperref}
\usepackage[capitalise]{cleveref}
\usepackage{graphicx}
\usepackage{bbm}
\usepackage{nameref}
\usepackage{verbatim}
\usepackage{csquotes}
\usepackage{natbib}
\usepackage{mathtools}
\usepackage{xcolor}
\usepackage{aliascnt}

\title{Rigidity phenomena for time changes of products of Anosov flows}
\author{Amadeus Maldonado} 
\address{Department of Mathematics, Northwestern University, Evanston, IL 60208, USA}
\email{amadeusmaldonado2027@u.northwestern.edu}

\author{Miri Son}
\address{Department of Mathematics, Rice University, Houston, TX 77005, USA}
\email{ms235@rice.edu}

\numberwithin{equation}{section}

\newtheorem{theorem}{Theorem}[section]

\newaliascnt{Lemma}{theorem}
\newtheorem{Lemma}[Lemma]{Lemma}
\aliascntresetthe{Lemma}

\newaliascnt{proposition}{theorem}
\newtheorem{proposition}[proposition]{Proposition}
\aliascntresetthe{proposition}

\newaliascnt{corollary}{theorem}

\aliascntresetthe{corollary}

\newaliascnt{definition}{theorem}
\newtheorem{definition}[definition]{Definition}
\aliascntresetthe{definition}

\newaliascnt{remark}{theorem}
\newtheorem{remark}[remark]{Remark}
\aliascntresetthe{remark}

\newtheorem{mainthm}{Theorem}

\newtheorem{maincor}[mainthm]{Corollary}

\Crefname{Lemma}{Lemma}{Lemmas}

\Crefname{theorem}{Theorem}{Theorems}

\Crefname{corollary}{Corollary}{Corollaries}

\Crefname{definition}{Definition}{Definitions}

\newcommand{\R}{\mathbb{R}}

\newcommand{\N}{\mathbb{N}}

\begin{document}

\begin{abstract}
    Our main results establish two rigidity phenomena in the class of time changes of a fixed product of Anosov flows.
    Our first result shows that two time changes having the same stabilizers for all periodic orbits are conjugate up to automorphism.
    The second rigidity result proves that if the stabilizers of periodic orbits can be simultaneously diagonalized, then the time change is conjugate to a product of flows up to automorphism.
    We allow our time changes to be H\"older continuous, which by structural stability implies that our results hold for $C^1$ perturbations of products of Anosov flows.

    We apply our main results to $C^1$ time changes of products of Anosov flows on $3$-dimensional manifolds.
    For such actions, we show that being totally Anosov, being conjugate to a product of flows and having the kernels of Lyapunov functionals not depend on the periodic orbits are all equivalent properties.
    We also build counterexamples to the Katok-Spatzier conjecture as time changes of products of any transitive Anosov flows, extending the main result of \cite{Vinhage} beyond the continuously accessible case.
\end{abstract}

\maketitle

\section{Introduction}\label{sec:introduction}

Let $\alpha\colon\R^k\curvearrowright M$ be a locally free $C^s$ action for $s\geq 1$ on a smooth manifold $M$, equipped with a Riemannian metric defining a norm $\|\cdot\|$ on the tangent bundle, and let $\mathcal{O}$ denote its orbit foliation. 
We say $a \in \mathbb{R}^k$ is an \emph{Anosov element} if there exists an $\R^k$-invariant splitting of the tangent bundle $TM=E^u_a\oplus E^s_a\oplus T\mathcal{O}$ into nontrivial subbundles and constants $\lambda,C>0$ such that for all $t>0$,
    \begin{equation*}
        \|D\alpha(ta)|_{E^s_a}\|\leq Ce^{-\lambda t}\quad \text{and}\quad \|D\alpha(-ta)|_{E^u_a}\|\leq Ce^{-\lambda t}.
    \end{equation*}
    We say that the action $\alpha$ is \emph{Anosov} if it has at least one Anosov element. If the set of Anosov elements is dense in $\R^k$, $\alpha$ is \emph{totally Anosov}.
    The dimension $k$ is also called the \emph{rank} of the action.
    It is a \emph{flow} if $k=1$ and a \emph{higher rank action} if $k \geq 2$.

    Higher rank Anosov actions often display many rigidity properties, in contrast to their rank $1$ counterparts.
A series of papers by Katok and Spatzier have studied several such actions arising from  algebraic settings.
These actions enjoy cocycle rigidity (\cite{KatokSpatzier}), measure rigidity (\cite{KatokSpatziermeasure}) and local smooth rigidity (\cite{KSsmoothrigid}).
On the other hand, products of Anosov flows, although technically higher rank, will inherit non rigid properties coming from its rank $1$ parts.

This led to the famous Katok--Spatzier conjecture, which claims that, in the absence of rank $1$ factors, higher rank Anosov actions are algebraic.
The conjecture was proven by Spatzier and Vinhage in  \cite{SpatzierVinhage} with the added assumptions of totally Cartan and cone transitivity.
However, the conjecture as originally stated is false.
In \cite{Vinhage}, Vinhage builds a family of Anosov $\mathbb{R}^2$ actions which have no rank $1$ factors and can not be made algebraic.
He does so by constructing nontrivial time changes of products of Anosov flows, the main subject of this paper.

\subsection{Main results}

Let us first fix some notation and definitions.
Let $\alpha \colon \mathbb{R}^k \curvearrowright M $ and $\beta \colon \mathbb{R}^k \curvearrowright N$ be locally free $C^{s}$ actions on compact smooth manifolds $M$ and $N$ for some $1 \leq s \leq \infty$.
For $0 \leq r \leq s$, we say that $\alpha$ and $\beta$ are \emph{$C^r$ conjugate up to automorphism} if there exists a $C^r$ homeomorphism $h \colon M \to N$ with $C^r$ inverse and $C \in GL(k,\mathbb{R})$ satisfying $h(\alpha(a)x) = \beta(Ca)h(x)$ for all $a \in \mathbb{R}^k$ and $x \in M$.
If $C = I$, we say they are \emph{conjugate}.

We now define a time change action, which has the same orbit foliation as the original action but a different parameterization along the orbit. 

\begin{definition}
    Consider two locally free $\R^k$-actions $\alpha_0$ and $ \alpha$ on a manifold $M$. We say that an action $\alpha$ is a $C^s$ time change of $\alpha_0$ if there exists a $C^s$ map $\varphi\colon\R^k\times M\rightarrow\R^k$ defined by $(a,x)\mapsto\varphi_x(a)$ such that $$\alpha(a)x=\alpha_0(\varphi_x(a))x,$$ and for each $x\in M$, the map $\varphi_x(\cdot)$ is a $C^s$ homeomorphism with $C^s$ inverse of $\R^k$. We say that $\alpha$ is a constant time change of $\alpha_0$ if there exists $C\in \mathrm{GL}(k,\R)$ such that $$\alpha(a)x=\alpha_0(Ca)x.$$
\end{definition}

Let $f_i$ be $C^s$ transitive Anosov flows on compact manifolds $X_i$. Consider the $\R^k$-action $\alpha_0=\Pi_{i=1}^k f_i$ on the product manifold $M=\Pi_{i=1}^k X_i$ defined by $$\alpha_0(t_1,...,t_k)(x_1,...,x_k)=(f_1^{t_1}x_1,...,f_k^{t_k}x_k).$$
An element $a = (t_1,...,t_k) \in \mathbb{R}^k$ is Anosov for $\alpha_0$ if, and only if, $t_i \neq 0$ for all $i =1,...,k$.

We first establish periodic orbit stabilizer rigidity for time changes of $\alpha_0$.
This may be viewed as a higher rank analogue of the marked length spectrum rigidity. In the rank one case $k=1$, since we start with a time change action, the conclusion follows directly from Livshitz. In contrast, the higher rank case is not a direct consequence. Further details are given in \cref{sec:cocycle}.

A point $p\in M$ is called \emph{periodic} for an $\R^k$-action $\alpha$ if its orbit $\alpha(\R^k)p$ is closed. We denote by $\text{Per}(\alpha)$ the set of periodic points of  $\alpha$. 
For each periodic point $p\in M$ of a locally free action $\alpha$, the stabilizer $\text{Stab}_{\alpha}(p)$ is a lattice in $\R^k$. Let $\Delta$ denote the set of lattices in $\R^k$. Define the \emph{marked stabilizer spectrum} function of $\alpha$ $$l_{\alpha}\colon \text{Per}(\alpha)\rightarrow\Delta,$$ by assigning to each periodic point $p$ its stabilizer $\text{Stab}_{\alpha}(p)$.
If $\alpha$ is a time change of $\beta$, note that $\text{Per}(\alpha) = \text{Per}(\beta)$. This implies $l_{\alpha}$ and $l_{\beta}$ are functions on the same space.

\begin{mainthm}
\label{thm1}
    Consider two $C^r$ time change actions $\alpha_1$ and $\alpha_2$ of $\alpha_0$, for $r\in(0,s]$. If there exists $C \in GL(k,\mathbb{R})$ such that $l_{\alpha_1} = C \cdot  l_{\alpha_2}$, then $\alpha_1$ is $C^{r_*}$ conjugate to $\alpha_2$ up to automorphism, where $r_* = r-1+\text{Lip}$ if $r \in \mathbb{N}\setminus \{1\}$ and $r_* = r$ otherwise.
\end{mainthm}

The regularity $C^{r_*}$ arises from the higher rank Livshitz \cref{Livshitz} which uses Journé's Lemma. Here $r_*=r-1+\text{Lip}$, means that the conjugacy is a $C^{r-1}$-diffeomorphism whose $(r-1)$-jet is Lipschitz.

One class of time changes of $\alpha_0$ that we may consider consists of those obtained by reparametrizing each flow $f_i$.
For actions $\alpha$ in this class and $p \in \text{Per}(\alpha_0)$, $\text{Stab}_{\alpha}(p)$ is generated by integer multiples of the standard basis elements $\{e_1,...,e_k\} \subset \mathbb{R}^k$.
For constant time changes of such actions, we instead have some basis $\{u_1,...,u_k\} \subset \mathbb{R}^k$ such that the stabilizers of periodic orbits are generated by integer multiples of the basis elements.
The next Theorem shows that this simultaneous diagonalizability of stabilizers property classifies this class of time changes up to automorphism.

\begin{mainthm}
\label{thm2}
    Let $\alpha$ be a $C^r$ time change action of $\alpha_0$, for $r \in (0,s]$. Suppose there exist linearly independent vectors $u_1,...,u_k\in\R^k$ such that for any $p\in\text{Per}(\alpha_0)$, there exist scalars $a_i(p) \in \mathbb{R}\setminus \{0\}$ with $$\text{Stab}_\alpha(p)= a_1(p)u_1\mathbb{Z} \oplus ... \oplus a_k(p)u_k\mathbb{Z}.$$ Then $\alpha$ is $C^{r_*}$ conjugate up to automorphism to $g_1\times...\times g_k$, where each $g_i$ is a $C^{r_*}$ time change of $f_i$.
\end{mainthm}

The product of Anosov flows provides the appropriate setting for our results, rather than higher-rank algebraic one, since the latter is cocycle rigid \cite{KatokSpatzier} and hence admits no nontrivial time changes.

One natural source of time changes is orbit equivalences.
We say that two actions $\alpha,\beta \colon \mathbb{R}^k \curvearrowright M$ are \emph{$C^r$ orbit equivalent} if there exists a $C^r$ homeomorphism $h \colon M \to M$ such that $h(\mathcal{O}_{\beta}(x)) = \mathcal{O}_{\alpha}(h(x))$, where $\mathcal{O}_{\beta}$ and $\mathcal{O}_{\alpha}$ are the orbit foliations of $\beta$ and $\alpha$ respectively.
In this case, define $\tilde{\beta}(a)x = h(\beta(a)h^{-1}x)$.
Then $\tilde{\beta}$ is a $C^r$ action with the same orbits as $\alpha$, which can be viewed as a time change of $\alpha$ which is $C^r$ conjugate to $\beta$.
In particular, for $\alpha = \alpha_0$ and after appropriately writing $\tilde{\beta}$ as a $C^{r'}$ time change of $\alpha_0$ for some $r'>0$, \cref{thm1} and \cref{thm2} are applicable to actions orbit equivalent to $\alpha_0$.

We say that $\alpha$ is \emph{structurally stable} if any action $\beta \colon \mathbb{R}^k \curvearrowright M$ $C^1$ close to $\alpha$ is orbit equivalent to $\alpha$.
As a consequence of Theorem 7.1 of \cite{HPS}, Anosov actions are H\"older structurally stable.
This means that any $C^1$ perturbation of $\alpha_0$ can be written, up to H\"older conjugacy, as a H\"older time change of $\alpha_0$ and \cref{thm1} and \cref{thm2} also hold for actions in a $C^1$ neighborhood of $\alpha_0$. 
The drawbacks are that the conjugacy will not necessarily be $C^1$ and the action product action constructed on \cref{thm2} is only H\"older.

\begin{maincor}
There exists a $C^1$ neighborhood $\mathcal{U}$ of $\alpha_0$ such that
\begin{enumerate}
    \item For $\alpha_1,\alpha_2 \in \mathcal{U}$, denote by $h_i \colon M \to M$ the H\"older orbit equivalence between $\alpha_0$ and $\alpha_i$ for $i=1,2$.
    If there exists $C \in GL(k,\mathbb{R})$ such that $\text{Stab}_{\alpha_1}(h_1(p)) = C \cdot \text{Stab}_{\alpha_2}(h_2(p))$, for all $p \in \text{Per}(\alpha_0)$,
    then $\alpha_1$ is H\"older conjugate to $\alpha_2$, up to automorphism;
    \item If for $\alpha \in \mathcal{U}$, there exist linearly independent vectors $u_1,...,u_k \in \mathbb{R}^k$ such that for any $p \in \text{Per}(\alpha)$, there exists scalars $a_i(p) \in \mathbb{R}$ satisfying $\text{Stab}_{\alpha}(p) = a_1(p)u_1\mathbb{Z} \oplus ... \oplus a_k(p)u_k\mathbb{Z}$, then $\alpha$ is H\"older conjugate up to automorphism to $g_1\times... \times g_k$ where each $g_i$ is a H\"older time change of $f_i$.
\end{enumerate}
\end{maincor}

In certain cases in the rank $1$ setting, H\"older conjugacies between smooth Anosov flows can be upgraded to a smooth conjugacies. For example, such an upgrade is possible for transitive Anosov flows on a $3$-manifold with the same periodic orbit Lyapunov data, for contact Anosov flows under bunching assumptions, or for volume preserving Anosov flows on a $3$-manifold, see \cite{LM,GRH,GRH2, GRH3,GLRH}.
It would be interesting to see in which cases the H\"older conjugacies obtained by \cref{thm1} and \cref{thm2} can be upgraded to smooth conjugacies.

\subsection{Applications to Cartan actions}
Given an Anosov action $\alpha \colon \mathbb{R}^k \curvearrowright M$, consider the decomposition $TM = T\mathcal{O}_{\alpha}\oplus \bigoplus_{i \in I}E_i$ into coarse Lyapunov subbundles. A more precise definition and further details can be found in \cref{sec:classifyAnosov}. We say that $\alpha$ is \emph{Cartan} if each $E_i$ is one dimensional for all $i\in I$.

Suppose that each $X_i$ is $3$-dimensional. This implies that $\alpha_0 = f_1 \times ... \times f_k$ is a Cartan action on $M = X_1 \times ... \times X_k$.
Write $TM = T \mathcal{O}_{\alpha} \oplus \bigoplus_{i \in I}E^{\alpha_0}_i$.
If $\alpha \colon \mathbb{R}^k \curvearrowright M$ is a $C^1$ time change of $\alpha_0$ and $p \in \text{Per}(\alpha)$, there exists a splitting $T_pM = T_p\mathcal{O} \oplus \bigoplus_{i \in I} E^{\alpha}_i$ which is $D_p\alpha(a)$ invariant for every $a \in \text{Stab}_{\alpha}(p)$ (see \cref{changeofcoord}).
We may then consider the Lyapunov functionals $\lambda^{i,\alpha}_p (a) = \int \log \|D_x\alpha(a)|_{E^\alpha_i}\| d\mu_p(x)$ where $\mu_p$ is the ergodic measure supported on the orbit of $p$. 

The next Theorem relates the property of $\alpha$ being a product with the totally Anosov property and the Lyapunov functionals along periodic orbits.

\begin{mainthm}
    \label{thm4}
    Let $f_i \colon X_i \to X_i$ be transitive Anosov flows on a compact, connected Riemannian manifolds of dimension 3 for $i=1,...,k$, $\alpha_0 = f_1 \times ... \times f_k$ and  
    $\alpha$ be a $C^1$ time change of $\alpha_0$. The following are equivalent:
    \begin{enumerate}
        \item $\alpha$ is $C^1$ conjugate up to automorphism to a product $g_1 \times ... \times g_k$ where each $g_i$ is a $C^1$ time change of $f_i$;
        \item $\alpha$ is totally Anosov;
        \item For every $i \in I$ and $p,q \in \text{Per}(\alpha)$, $\ker \lambda_{p}^{i,\alpha} = \ker \lambda_{q}^{i,\alpha}$.
    \end{enumerate}
\end{mainthm}

Using the formula of how Lyapunov functionals change under $C^1$ time changes  in \cref{changeofcoordLyap} together with \cref{thm4}, we are able to extend the main result of \cite{Vinhage} for products of any transitive Anosov flows, not only those satisfying the continuously accessible assumption (see Definition 3.4 of \cite{Vinhage}).
We refer to \cite{Vinhage} for the definitions of a rank one factor and a homogeneous action.

\begin{mainthm}
\label{counterexamples}
    Let $X$ and $Y$ be $3$-dimensional, compact, connected, Riemannian manifolds admitting transitive Anosov flows.
    Then, there exists a $C^{\infty}$ cone transitive action $\alpha \colon \mathbb{R}^2 \curvearrowright X \times Y$ which is Anosov, has no $C^1$ rank one factors and is not homogeneous.
\end{mainthm}

\subsection{Structure of paper}

In \cref{sec:cocycle} we start by proving some basic, although important, abelian cocycle facts.
One such fact that we wish to highlight is \cref{invertiblecocycle} which states that time changes of $\mathbb{R}^k$ actions are in one to one correspondence with the notion of invertible $\mathbb{R}^k$ valued cocycles.
This means that we can view \cref{thm1} and \cref{thm2} as Theorems about invertible cocycles over $\alpha_0$.

In \cref{sec:decomp} we prove \cref{decompprop} which states that, mod coboundaries, cocycles over $\alpha_0$ can be decomposed into sum of cocycles over the Anosov flows $f_i$.

Next, \cref{thm1} is proved in \cref{sec:thm1} and \cref{thm2} in \cref{sec:thm2}. Their proofs follow similar arguments using the decomposition \cref{decompprop} and higher rank Livshitz \cref{Livshitz}.
While in \cref{thm1} the cocycles defining the time changes are given, \cref{sec:thm2} has the added difficulty of needing to construct invertible cocycles with the right periodic data.
This uses \cref{softmaneLemma} which constructs good representatives in the cohomology class of a given cocycle.

Finally, \cref{sec:classifyAnosov} is dedicated to proving \cref{thm4} and \cref{counterexamples}.
They rely on \cref{lyapexpthm}, which classifies the set of Anosov elements for Cartan actions in terms of the Lyapunov functionals, and a \cref{changeofcoordLyap} which says how Lyapunov functionals change under $C^1$ time changes.

\section{Cocycles}\label{sec:cocycle}
In this section, we establish several properties of cocycles that will play key roles in the proofs of the main theorems. Let $\varphi\colon \mathbb{R}^k \times M \to \mathbb{R}^l$ be a $cocycle$ over the action $\alpha\colon\R^k\curvearrowright M$ defined by $(a,x)\mapsto\varphi_x(a)$ satisfying the cocycle relation
\begin{equation}\label{cocycle}
    \varphi_x(a+b)=\varphi_{\alpha(a)x}(b)+\varphi_x(a).
\end{equation}
We say that the cocycle $\varphi$ is a \emph{coboundary} if there exists $H \colon M \to \mathbb{R}^l$ satisfying $\varphi_x(a) = H(\alpha(a)x)-H(x)$.
In this case, we say $H$ is a \emph{transfer function}.
Given $p \in M$ periodic and $a,b \in \text{Stab}_{\alpha}(p)$, \eqref{cocycle} implies
\begin{equation*}
    \varphi_p(a+b) = \varphi_p(a) +\varphi_p(b).
\end{equation*}
Since $\text{Stab}_{\alpha}(p) < \mathbb{R}^k$ is a lattice when the action $\alpha$ is locally free, we may extend $\varphi_p \colon \text{Stab}_{\alpha}(p) \to \mathbb{R}^l$ to a linear map $\mathcal{P}^{\varphi}_p \colon \mathbb{R}^k \to \mathbb{R}^l$.
Furthermore, for $c \in \mathbb{R}^k$ and $a \in \text{Stab}_{\alpha}(p)$, \eqref{cocycle} implies
\begin{equation*}
    \varphi_{\alpha(c)p}(a) = \varphi_p(a+c)-\varphi_p(c)
    =\varphi_p(a).
\end{equation*}
This means that $\mathcal{P}_p^{\varphi}$ only depends on the orbit of $p$. The map $\mathcal{P}^{\varphi}_p$ can equivalently be defined as
\begin{equation}
    \label{Periodmap}
    \mathcal{P}^{\varphi}_p(a) = \int \varphi_x(a)d\mu_p(x),
\end{equation}
for $a \in \mathbb{R}^k$ where $\mu_p$ is the periodic invariant measure on the orbit of $p$.
Indeed, the expression defined by \eqref{Periodmap} is linear in $a \in \mathbb{R}^k$ by the cocycle property of $\varphi$ and invariance of $\mu_p$ and it agrees with $\varphi_p$ on $\text{Stab}_{\alpha}(p)$.

When $M$ is compact and the cocycle is continuous, we first observe that the norm $\|\mathcal{P}_p^{\varphi}\|$ is uniformly bounded in $p$.
Indeed, by \eqref{Periodmap},
\begin{equation}
    \label{bddnorm}
    \underset{p \in \text{Per}(\alpha)}{\sup}\|\mathcal{P}_p^{\varphi}\|
    \leq \underset{p \in \text{Per}(\alpha)}{\sup} \int \underset{|a|\leq 1}{\sup}|\varphi_x(a)|d\mu_p(x)
    \leq \underset{x \in M, |a| \leq 1}{\sup}|\varphi_x(a)| < \infty.
\end{equation}

\subsection{Higher rank Livshitz Theorem}
In this subsection, we establish a Livshitz theorem for higher rank cone transitive Anosov actions. 
An Anosov action $\alpha$ is $transitive$ if there exists a point with a dense $\R^k$-orbit.  
We say that an Anosov action is \emph{cone transitive} if there exists an open cone $C\subset\R^k$ and a point $x\in M$ such that $\alpha(C)x$ is dense, and the only non-Anosov element of $\bar{C}$ is $0$. 
The action $\alpha_0$ defined in \cref{sec:introduction} is cone transitive.
One way to see this is by noting that it has a dense set of periodic orbits, which is equivalent to cone transitivity for totally Anosov actions by \cite[Lemma 4.17]{SpatzierVinhage}. 

\begin{Lemma}[\cite{SpatzierVinhage}, Lemma 4.17]\label{conetransitive}
    Let $\alpha\colon\R^k\curvearrowright M$ be a cone transitive, $C^k$, Anosov action. Then, the action has a dense set of periodic orbits. Moreover, there exists an Anosov element $a$ with a dense forward orbit $\{\alpha(na)x\colon n\in \mathbb{N}\}$.
\end{Lemma}
\begin{remark}
    In the original version of Lemma 4.17 in \cite{SpatzierVinhage}, 
    there are more properties with an additional assumption that the action is totally Anosov. In our setting, we assume only that the action is Anosov, rather than totally Anosov, and the above lemma remains valid under this weaker hypothesis, following the same argument.
\end{remark}

The following is a higher rank version of the Livshitz theorem for higher rank Anosov actions. The proof for the Hölder cocycle closely follows an argument in Theorem 2.10 of \cite{KatokSpatzier}, but without the volume preserving assumption and with cone transitivity. We next upgrade the regularity of Hölder transfer functions using Journé's lemma. In particular, when a coboundary $\varphi$ is $C^r$ for $r \in (0,\infty]$, we obtain a $C^{r_*}$ transfer function $H$. Here $r_* = r-1+\text{Lip}$ if $r \in \mathbb{N}\setminus \{1\}$ and $r_* = r$ otherwise. The argument proceeds in two steps. First, we prove that a Hölder transfer function $H$ is $C^r$ along the orbit direction, stable and unstable foliations of an Anosov element. Then we apply Journé's lemma \cite{Journe} to deduce global regularity. A similar application of Journé's lemma appears in \cite[Theorem B]{GRH2} of Gogolev and Rodriguez-Hertz.

\begin{theorem}[Livshitz theorem]
\label{Livshitz}
    Let $\alpha \colon \mathbb{R}^k \curvearrowright M$ be a cone transitive, Anosov action on a closed manifold $M$ and $\varphi \colon \mathbb{R}^k \times M \to \mathbb{R}^l$ be a H\"older cocycle over $\alpha$.
    If $\mathcal{P}^{\varphi}_p =0 $, for all $p \in M$ periodic, then $\varphi$ is a H\"older coboundary.
    Furthermore, if $\varphi$ is $C^r$ for $r \in (0,\infty]$, then the transfer function is $C^{r_*}$.
\end{theorem}

\begin{proof}
    Since the action $\alpha$ is cone transitive, there exist an Anosov element $a\in \R^k$ and a point $x\in M$ with a dense forward orbit $\{\alpha(na)x\colon n\in\N\}$ in $M$ by \cref{conetransitive}. Define a function $H\colon M\to\R^l$ on the dense subset by $$H(\alpha(na)x)=\varphi_{x}(na)$$ for each $n\in\N$. By Lemma 4.8 in \cite{KatokSpatzier}, $H$ is Hölder continuous with the same Hölder exponent as $\varphi$, and it extends uniquely to a Hölder function on all of $M$. Moreover, $H(\alpha(a)x)-H(x)$ is also Hölder. Since $H(\alpha(a)x)-H(x)=\varphi_x(a)$ holds on the dense subset of $M$, continuity implies that it holds on all of $M$.

    Define the cocycle $\psi_x(b) = \varphi_x(b)-H(\alpha(b)x)+H(x)$.
    It satisfies for all $x \in M$ and $b \in \mathbb{R}^k$
    \begin{equation*}
        \psi_{x}(a+b) = \psi_x(b) + \psi_{\alpha(b)x}(a) = \psi_{x}(a)+\psi_{\alpha(a)x}(b).
    \end{equation*}
    Since $\psi_{y}(a) = 0$ for any $y \in M$,
    \begin{equation*}
        \psi_x(b) = \psi_{\alpha(a)x}(b).
    \end{equation*}
    Since $\alpha(a)$ has a dense orbit, $x \mapsto \psi_x(b)$ is constant.
    However, for $p \in \text{Per}(\alpha)$, denoting by $\mu_p$ the $\alpha$ invariant measure supported on the orbit of $p$,
    \begin{equation*}
        \int \psi_{x}(b) d\mu_p(x) = \mathcal{P}^{\varphi}_p(b) = 0.
    \end{equation*}
    Therefore, $\psi_p(b) = 0$ for any $p \in \text{Per}(\alpha)$.
    Since periodic orbits are dense and $x \mapsto \psi_x(b)$ is continuous, then $\psi_x(b) = 0$ for all $x \in M$ and $b \in \mathbb{R}^k$.
    Therefore, $\varphi$ is a Hölder coboundary. 

    Fix an Anosov element $a\in\R^k$. We first show that the transfer function $H$ is $C^r$ along the unstable foliation $\mathcal{W}^u_a$. Since $H(x)=H(\alpha(-a)x)-\varphi_x(-a)$, iterating gives 
    \begin{equation*}
    H(x)=H(\alpha(-na)x)-\sum_{j=0}^{n-1}\varphi_{\alpha(-ja)x}(-a).
    \end{equation*}
    Define $f_j(x):=\varphi_{\alpha(-ja)x}(-a).$ 
    Now let $y\in\mathcal{W}^u_a(x).$ Since $\alpha(-a)$ contracts unstable leaves exponentially and $H$ is Hölder continuous, $$H(\alpha(-na)x)-H(\alpha(-na)y)\rightarrow0$$ as $n\rightarrow0$. Hence, $$H(y)-H(x)=-\sum_{j=0}^{\infty}f_j(y)-f_j(x).$$ Therefore, it suffices to prove that the series
    $$y \mapsto \sum_{j=0}^\infty f_j(y) - f_j(x)$$ converges in the $C^r$-topology along unstable leaves.
    Each function $f_j$ is $C^r$ along unstable leaves since $\varphi$ is $C^r$ and the unstable foliation is $\alpha$ invariant with $C^r$ leaves of $\mathcal{W}_a^u$. We first estimate the derivatives of $f_j.$
    Denote by $\varphi^{-a} = \varphi(-a,\cdot)$.
    Since $f_j=\varphi^{-a}\circ\alpha(-ja)$, the chain rule gives 
    $$D_yf_j[v]=D_{\alpha(-ja)y}\varphi^{-a}[D_y\alpha(-ja)[v]].$$ Along unstable leaves, $D\alpha(-ja)$ decays exponentially and $D\varphi^{-a}$ is uniformly bounded. Thus $$\|Df_j|_{\mathcal{W}_a^u}\|\leq C\lambda^j$$ for some constants $C>0$ and $0<\lambda<1.$

    For the second derivative,
    \begin{align*}
        D_y^2f_j[u,v]=& D^2_{\alpha(-ja)y}\varphi^{-a}[D_y\alpha(-ja)[u],D_y\alpha(-ja)[v]]\\&+D_{\alpha(-ja)y}\varphi^{-a}[D_y^2\alpha(-ja)[u,v]].
    \end{align*}
    The first term already decays exponentially. Thus $D^2_yf_j$ decays exponentially if $D^2\alpha(-ja)|_{\mathcal{W}_a^u}$ decays exponentially. 
    Iterating the chain rule gives
    \begin{align*}
        D^2_x\alpha(-ja)=
        \sum_{k=0}^{j-1}&D_{\alpha(-(k+1)a)x}\alpha(-(j-1-k)a)\\&\circ D^2_{\alpha(-ka)x}\alpha(-a)[D_x\alpha(-ka),D_x\alpha(-ka)].
    \end{align*}
    Since $D\alpha(-ma)$ contracts exponentially along unstable leaves and $D^2\alpha(-a)$ is uniformly bounded, each term is bounded by $C'\lambda^{j-1+k}$ for some constant $C'>0$. We obtain $$\|D^2\alpha(-ja)\|\leq \sum_{k=0}^{j-1}C'\lambda^{j-1+k}\leq C''\lambda^{j-1}$$ for some constant $C''>0.$ Consequently, $D^2f_j$ decays exponentially. 

    The same argument works for higher derivatives. By the chain rule, each $m$-derivative of $f_j$ along the unstable foliation can be written as a finite sum of terms of the form 
    \begin{equation*}
        D^k_{\alpha(-ja)y}\varphi^{-a}[D^{i_1}_y\alpha(-ja),...,D^{i_k}_y\alpha(-ja)],
    \end{equation*}
    where the integer $1\leq k\leq m$ and $i_1+...+i_k=m$ with each $i_l\geq1.$ 
    Since the derivatives of $\varphi$ are uniformly bounded and all derivatives of $\alpha(-ja)$ decay exponentially along unstable leaves, it follows inductively that $D^mf_j$ is bounded by a constant times $\lambda^j$ for every integer $m\leq r$.
    
    It remains to show that $D^{\lfloor r\rfloor}f_j$ is $\theta$-Hölder along unstable leaves for $0<\theta\leq r-\lfloor r\rfloor$. 
    For $y,z\in\mathcal{W}^u_a(x)$, $D^{\lfloor r\rfloor}f_j(y)-D^{\lfloor r\rfloor}f_j(z)$ is a finite sum of terms of the form 
    \begin{align*}
        &D^{k}_{\alpha(-ja)y}\varphi^{-a}[D^{i_1}_y\alpha(-ja),...D^{i_k}_y\alpha(-ja)] \\ 
        &-D^{k}_{\alpha(-ja)z}\varphi^{-a}[D^{i_1}_z\alpha(-ja),...D^{i_k}_z\alpha(-ja)],
    \end{align*}
    where the integer $1\leq k\leq\lfloor r\rfloor$, $i_1+...+i_k=\lfloor r\rfloor$, and each $i_l\geq 1.$ Each such term can be decomposed as
    \begin{align*}
        &(D^{k}_{\alpha(-ja)y}\varphi^{-a}-D^{k}_{\alpha(-ja)z}\varphi^{-a})[D^{i_1}_y\alpha(-ja),...D^{i_k}_y\alpha(-ja)] \\
        &+D^{k}_{\alpha(-ja)z}\varphi^{-a}([D^{i_1}_y\alpha(-ja),...D^{i_k}_y\alpha(-ja)]-[D^{i_1}_z\alpha(-ja),...D^{i_k}_z\alpha(-ja)]).
    \end{align*}
    Since $D^{k}\varphi$ is $\theta$-Hölder and $\alpha(-a)$ contracts unstable leaves that decay exponentially, the first term is bounded by
    \begin{align*}
        C\operatorname{d}(\alpha(-ja)y,\alpha(-ja)z)^\theta\prod_{l=1}^k\|D^{i_l}_y\alpha(-ja)\|.
    \end{align*}
    Using $\operatorname{d}(\alpha(-ja)y,\alpha(-ja)z)\leq C\lambda^j\operatorname{d}(y,z)$ and $\alpha(-a)$ contracts unstable leaves that decay exponentially, this is bounded by
    \begin{align*}
        C\lambda^{j\theta}\operatorname{d}(y,z)^\theta\lambda^{jk}\leq C\lambda^{j}\operatorname{d}(y,z)^\theta,
    \end{align*}
    after changing $\lambda$. 
    Similarly, since $D\alpha(-ja)$ is $\theta$-Hölder along unstable leaves, the second term is also bounded by $C\lambda^{j}\operatorname{d}(y,z)^\theta$. Since there are only finitely many terms, we obtain $D^{\lfloor r\rfloor}f_j$ is $\theta$-Hölder along unstable leaves.
    
    Therefore the series $\sum_j (f_j(\cdot)-f_j(x))$ converges in $C^r$ topology along unstable leaves, and therefore $H$ is $C^r$ along $\mathcal{W}^u_a$. The proof for the stable foliation $\mathcal{W}^s_a$ is analogous. Thus $H$ is $C^r$ along both the stable and unstable foliations.
    It is clear that $H$ is $C^r$ along the orbit direction since the coboundary equation $H(\alpha(a)x) = \varphi_x(a)+H(x).$

    Applying Journé's lemma to the transverse unstable and orbit foliations, we have $H$ is $C^{r_*}$ along the weak unstable foliation. We then apply Journé's lemma again to the transverse weak unstable and stable foliations, we obtain $H$ is $C^{r_*}$ on $M$.
\end{proof}

\subsection{Time changes}
Consider an action $\alpha$ that is a time change of a locally free action $\alpha_0 \colon \mathbb{R}^k \curvearrowright M$. 
Let $\varphi \colon \mathbb{R}^k \times M \to \mathbb{R}^k$ be the function satisfying 
\begin{equation*}
    \alpha_0(a)x=\alpha(\varphi_x(a))x \quad \text{for}~ \text{all}~ a\in\R^k, x\in M.
\end{equation*}
By Lemma 2.3 of \cite{Vinhage}, $\varphi$ is a cocycle over $\alpha_0$.
For $x \in M$, it follows from the definition that
\begin{equation*}
    \mathcal{P}^{\varphi}_p(\text{Stab}_{\alpha_0}(p)) = \text{Stab}_{\alpha}(p).
\end{equation*}

Recall that for $x \in M$, $\varphi_x \colon \mathbb{R}^k \to \mathbb{R}^k$ is a homeomorphism and we call such cocycles \emph{invertible}. Conversely, if an invertible cocycle is given, we can obtain a time change action from it, as stated in the following Lemma which follows from the proof of Lemma 2.4 of \cite{Vinhage}.

\begin{Lemma}[\cite{Vinhage}, Lemma 2.4] \label{invertiblecocycle}
Let $\varphi\colon\R^k\times M\to\R^k$ be an invertible cocycle over an action $\alpha_0\colon\R^k\curvearrowright M$. Then there exists a time change action $\alpha$ of $\alpha_0$ determined by $\varphi$.
\end{Lemma}

Consider two invertible cocycles $\varphi, \psi\colon\R^k\times M\to\R^k$ over an action $\alpha_0\colon\R^k\curvearrowright M$. By the Lemma above, there exist two time change actions $\alpha_\varphi$ and $\alpha_\psi$ of $\alpha_0$ determined by $\varphi$ and $\psi$, respectively. For all $a\in\R^k$ and $x\in M$, 
\begin{align*}          
    &\alpha_0(a)x=\alpha_\varphi(\varphi_x(a))x, \\
    &\alpha_0(a)x=\alpha_\psi(\psi_x(a))x.
\end{align*}
Since $\varphi$ is an invertible cocycle, there is an inverse cocycle $\varphi^{-1}$ satisfying
\begin{equation*}
    \varphi^{-1}_x(\varphi_x(a))=a \quad \text{and}\quad \varphi_x(\varphi_x^{-1}(a))=a.
\end{equation*}

\begin{Lemma}\label{coboundary}
    If $\varphi-\psi$ is a $C^r$ coboundary over $\alpha_0$, then $\alpha_\varphi$ is $C^r$ conjugate to $\alpha_\psi$.
\end{Lemma}
\begin{proof}
    Since $\varphi-\psi$ is a $C^r$ coboundary, there exists a $C^r$ map $H\colon M\to\R^k$ such that $\varphi_x(b)-\psi_x(b)=H(\alpha_0(b)x)-H(x)$, for all $b\in\R^k$ and $x\in M$. Setting $b=\varphi^{-1}_x(a)$, we obtain $a-\psi_x(\varphi^{-1}_x(a))=H(\alpha_\varphi(a)x)-H(x)$. By the definitions of $\alpha_\varphi$ and $\alpha_\psi$, and since the actions are abelian, we have
    \begin{align*}
        \alpha_\varphi(a)x 
        &= \alpha_0(\varphi^{-1}_x(a))x\\
        &= \alpha_\psi(\psi_x(\varphi^{-1}_x(a)))x \\
        &= \alpha_\psi(a-H(\alpha_\varphi(a)x)+H(x))x \\
        &= \alpha_\psi(H(\alpha_\varphi(a)x))^{-1}\alpha_\psi(a)\alpha_\psi(H(x))x.
    \end{align*}
    Therefore, $\alpha_\psi(H(\alpha_\varphi(a)x))\alpha_\varphi(a)x=\alpha_\psi(a)\alpha_\psi(H(x))x$. Defining a $C^r$ map $\pi\colon M\to M$ by $\pi(x)=\alpha_\psi(H(x))x$, we obtain $\pi(\alpha_\varphi(a)x)=\alpha_\psi(a)\pi(x)$, which proves the Lemma.
\end{proof}

\subsection{Adapted cocycles}\label{manesection}
The goal of this subsection is to prove \cref{softmaneLemma} using standard averaging arguments.
It is used both to construct invertible cocycles in the proof of \cref{thm2} and to verify hyperbolicity of Anosov elements in the proof of \cref{lyapexpthm}.
We remark that sharp versions of \cref{softmaneLemma} have been shown for the rank $1$ case (see \cite{LTDiffeos} for $\mathbb{Z}$ actions and \cite{LopesThieullen} for $\mathbb{R}$ actions).

Let $M$ be a compact metric space and $\alpha \colon \mathbb{R}^k \curvearrowright M$ be a $C^s$ action, for some $s\geq0$. If $s\geq 1$, then we also assume that $M$ is a manifold.
Let $\varphi \colon \mathbb{R}^k \times M \to \mathbb{R}^l$ be a $C^r$ cocycle over $\alpha$ with $r \leq s$.
Given $\nu$ a probability measure on $\mathbb{R}^k$, define the $C^r$ cocycle $\varphi^{\nu}$ by
\begin{equation*}
    \varphi^{\nu}_x(a) = \int \varphi_{\alpha(c)x}(a)d\nu(c).
\end{equation*}
\begin{Lemma}
    \label{nucohomologous}
    The cocycle $\varphi^{\nu}$ is $C^r$ cohomologous to $\varphi$.
\end{Lemma}
\begin{proof}
    Define $H\colon M \to \mathbb{R}^l$ by
    \begin{equation*}
        H(x) = \int\varphi_x(c)d\nu(c).
    \end{equation*}
    Note that $H$ is $C^r$. For $a \in \mathbb{R}^k$ and $x \in M$,
    \begin{equation}
    \label{differenceH}
        H(\alpha(a)x)-H(x)
        = \int \varphi_{\alpha(a)x}(c)-\varphi_x(c)d\nu(c).
    \end{equation}
    By the cocycle property,
    \begin{equation*}
        \varphi_x(c) + \varphi_{\alpha(c)x}(a)
        = \varphi_x(a+c)
        = \varphi_x(a)+\varphi_{\alpha(a)x}(c).
    \end{equation*}
    Therefore,
    \begin{equation*}
        \varphi_{\alpha(a)x}(c) 
        = \varphi_{\alpha(c)x}(a)-\varphi_x(a)+\varphi_x(c).
    \end{equation*}
    Substituting the above in \eqref{differenceH},
    \begin{equation*}
        H(\alpha(a)x)-H(x)
        = \int \varphi_{\alpha(c)x}(a) d\nu(c) -\varphi_x(a),
    \end{equation*}
    which proves the Lemma.
\end{proof}

Suppose that $\varphi$ is a $\mathbb{R}$-valued cocycle.
Let $\mathcal{M}(\alpha)$ be the set of $\alpha$-invariant measures on $M$ and $\mathcal{M}_{\text{erg}}(\alpha)$ be the ergodic ones.
For $a \in \mathbb{R}^d$ denote by
\begin{equation}
\label{infsup}
    \underline{\varphi}(a) = \underset{\mu \in \mathcal{M}(\alpha)}{\inf}\int \varphi_x(a) d\mu(x)
    \ \text{and} \
    \overline{\varphi}(a) = \underset{\mu \in \mathcal{M}(\alpha)}{\sup}\int \varphi_x(a) d\mu(x).
\end{equation}
It follows from the ergodic decomposition that we could define \eqref{infsup} by only considering the ergodic measures.
\begin{Lemma}
\label{softmaneLemma}
    Let $\varepsilon>0$ and $a \in \mathbb{R}^k$.
    There exists a $C^r$ cocycle $\psi \colon \mathbb{R}^k \times M \to \mathbb{R}$ which is $C^r$ cohomologous to $\varphi$ satisfying for all $t>0$ and $x \in M$,
    \begin{equation*}
        (\underline{\varphi}(a)-\varepsilon)t
        < \psi_x(ta)
        <(\overline{\varphi}(a)+\varepsilon)t
    \end{equation*}
\end{Lemma}
\begin{proof}
    Denote by $\nu_n = \frac{1}{2^kn^k}\text{Leb}|_{[-n,n]^k}$.
    We first claim that there exists $N >0$ such that 
    \begin{equation}
    \label{ineqfora}
        \underline{\varphi}(a)-\varepsilon
        < \varphi^{\nu_N}_x(a)
        < \overline{\varphi}(a)+\varepsilon, \ \text{for all} \ x \in M.
    \end{equation}
    Otherwise, for each $n \in \mathbb{N}$, there exists $x_n \in M$ such that
    $\varphi^{\nu_n}_{x_n}(a) \notin 
    (\underline{\varphi}(a)-\varepsilon,\overline{\varphi}(a)+\varepsilon)$. Let $\mu_n = \int \delta_{\alpha(c)x_n}d\nu_n(c)$ be a sequence of measures on $M$.
    By the Krylov Bogolyubov Theorem, we may pass to a subsequence if necessary and assume $\mu_n \rightarrow \mu$, where $\mu$ is $\alpha$-invariant.
    This implies that
    \begin{equation*}
        \varphi^{\nu_n}_{x_n}(a)
        = \int \varphi_x(a)d\mu_n(x) \rightarrow \int \varphi_x(a) d\mu(x)
        \in [\underline{\varphi}(a),\overline{\varphi}(a)],
    \end{equation*}
    which is a contradiction and proves the claim.

    For $x \in M$ and $b \in \mathbb{R}^k$, define
    \begin{equation*}
        \psi_x(b) = 
        \int_0^1 \varphi^{\nu_N}_{\alpha(sa)x}(b) ds.
    \end{equation*}
    Note that $\psi = \left(\varphi^{\nu_N}\right)^\nu$ for $\nu$ normalized Lebesgue measure on the segment $[0,a] \subset \mathbb{R}^k$.
    Therefore, $\psi$ is $C^r$ cohomologous to $\varphi$ by \cref{nucohomologous}.
    By the cocycle property, $\varphi^{\nu_N}_{\alpha(ta)x}(b) = \varphi^{\nu_N}_x(ta+b)-\varphi^{\nu_N}_x(b)$.
    Taking $b = ta$,
    \begin{align*}
        \psi_x(ta) =& \int_0^1 \varphi^{\nu_N}_x(ta+sa)-\varphi^{\nu_N}_x(sa)ds\\
        =&\int_t^{1+t}\varphi^{\nu_N}_x(sa)ds -\int_0^1\varphi_x^{\nu_N}(sa)ds.
    \end{align*}
    It follows that
    \begin{equation*}
        \frac{d}{dt}\psi_x(ta) = 
        \varphi^{\nu_N}_x(a+ta)-\varphi^{\nu_N}_x(ta)
        = \varphi^{\nu_N}_{\alpha(ta)x}(a),
    \end{equation*}
    where we once again used the cocycle property.
    We may write
    \begin{equation*}
        \psi_x(ta) =
        \int_0^t \varphi^{\nu_N}_{\alpha(sa)x}(a)ds.
    \end{equation*}
    The Lemma now follows from \eqref{ineqfora}.
\end{proof}

\section{Decomposition of cocycles over product actions}\label{sec:decomp}
In this section we prove \cref{decompprop}, a cohomology result for products of transitive Anosov flows.
The proof presented relies on Livshitz's Theorem.
We set up some notation first.

For $\alpha \colon \mathbb{R}^k \curvearrowright M$ a $C^s$ action on a manifold $M$ and $0\leq r \leq s$, we define the \emph{first $\mathbb{R}^l$-valued cohomology of $\alpha$} to be the group of $\mathbb{R}^l$ valued $C^r$ cocycles mod $\mathbb{R}^l$ valued $C^r$ coboundaries and we denote it by $H^{1,r}(\alpha,\mathbb{R}^l)$.
For $\varphi,\psi \colon \mathbb{R}^k \times M \to \mathbb{R}^l$ $C^r$ cocycles over $\alpha$, we write $\varphi \approx^r \psi$ to mean they are cohomologous by a $C^r$ coboundary, that is, $\varphi$ and $\psi$ are the same element in $H^{1,r}(\alpha,\mathbb{R}^l)$.

\begin{proposition}
\label{decompprop}
    Let $f_i$ be a transitive Anosov flow on $X_i$ for $i=1,...,k$ and $r \in (0,s]$.
    Then
    \begin{equation*}
        H^{1,r}(f_1\times...\times f_k,\mathbb{R}^l) \cong \bigoplus_{i=1}^k H^{1,r}(f_i,\mathbb{R}^l).
    \end{equation*}
    Moreover, for equivalence classes $H^{1,r}(f_1\times...\times f_k,\mathbb{R}^l) \ni [\varphi] \approx^r [\varphi^1]+...+[\varphi^k] \in \bigoplus_{i=1}^k H^{1,r}(f_i,\mathbb{R}^l)$ and $r \leq t \leq s$,
    $[\varphi]$ has a $C^t$ representative if, and only if, each $[\varphi^i]$ has a $C^t$ representative.
\end{proposition}
\begin{proof}
Denote by $M = X_1 \times...\times X_k$ and $\alpha = f_1\times...\times f_k$.
Given $C^r$ cocycles $\varphi^i$ over $f_i$, we can define a $C^r$ cocycle $\varphi = \varphi^1+...+\varphi^k$ over $\alpha$ by
\begin{equation*}
    \varphi_{(x_1,...,x_k)}(t_1,...,t_k) = 
    \sum_{j=1}^k \varphi^j_{x_j}(t_j).
\end{equation*}
This defines an injective homomorphism $\bigoplus_{i=1}^kH^{1,r}(f_i,\mathbb{R}^l) \to H^{1,r}(f_1\times ... \times f_k,\mathbb{R}^l)$.

We next prove it is surjective.
That is, given an abelian $C^r$ cocycle $\varphi \colon \mathbb{R}^k \times M \to \mathbb{R}^l$, we need to construct $\varphi^i \colon \mathbb{R}\times X_i \to \mathbb{R}^l$ $C^r$ cocycles over $f_i$ such that $\varphi$ is $C^r$ cohomologous to $\varphi^1 +...+\varphi^k$.
    Fix $(\overline{x}_1,...,\overline{x}_k) \in M$.
    For $i \in \{1,...,k\}$, $x_i \in X_i$ and $t_i \in \mathbb{R}$, 
    define
    \begin{equation}
    \label{coordinatecoc}
    \varphi^i_{x_i}(t_i) = 
        \varphi_{(\overline{x}_1,...,\overline{x}_{i-1},x_i,\overline{x}_{i+1},...,\overline{x}_k)}(t_ie_i).
    \end{equation}
    Since $\varphi$ is a $C^r$ cocycle over the product action $\alpha$,
    $\varphi^i$ is a $C^r$ cocycle over $f_i$.

    We will use Livshitz's theorem in \cref{Livshitz} to show that $\varphi$ is cohomologous to $\varphi^1+...+\varphi^k$.
    In order to do that, we need the following
    \begin{Lemma}
    \label{doesnotdepend}

        Under the hypothesis of \cref{decompprop}, let $p_i \in X_i$ be periodic for $f_i$ with period $t_i$.
        The expression
        \begin{equation*}
            \varphi_{(x_1,...,x_{i-1},p_i,x_{i+1},...,x_k)}(t_ie_i)
        \end{equation*}
        does not depend on $(x_1,...,x_{i-1},x_{i+1},...,x_k) \in \ X_1\times ... \times \widehat{X}_i \times ... X_k$.
    \end{Lemma}
\begin{proof}
    Fix $j \in \{1,...,k\}\setminus\{i\}$ and $\hat{x}_u \in X_u$ for $u \in \{1,...,k\}\setminus\{i,j\}$.
    Define $g_j \colon X_j \to \mathbb{R}^l$ by
    \begin{equation*}
        g(x_j) = \varphi_{(\hat{x}_1,...,\hat{x}_{j-1},x_j,\hat{x}_{j+1},...,\hat{x}_{i-1},p_i,\hat{x}_{i+1},...,\hat{x}_k)}(t_ie_i).
    \end{equation*}
    For $s \in \mathbb{R}$, the cocycle property gives
    \begin{align*}
        &\varphi_{(\hat{x}_1,...,x_j,...,p_i,,...,\hat{x}_k)}(t_ie_i+se_j) \\
        =& \varphi_{(\hat{x}_1,...,x_j,...,p_i,...,\hat{x}_k)}(t_ie_i) + \varphi_{(\hat{x}_1,...,x_j,...,f_i^{t_i}(p_i),...,\hat{x}_k)}(se_j)\\
        =& \varphi_{(\hat{x}_1,...,x_j,...,p_i,...,\hat{x}_k)}(se_j) + \varphi_{(\hat{x}_1,...,f_j^{s}(x_j),...,p_i,...,\hat{x}_k)}(t_ie_i),
    \end{align*}
    where ommited coordinates are hat coordinates.
    Since $t_i$ is a period of $p_i$ under $f_i$, the terms $\varphi_{(\hat{x}_1,...,x_j,...,f_i^{t_i}(p_i),...,\hat{x}_k)}(se_j)$ and $\varphi_{(\hat{x}_1,...,x_j,...,p_i,...,\hat{x}_k)}(se_j)$ above cancel out, which implies
    \begin{equation*}
        g(f_j^s(x_j)) = g(x_j).
    \end{equation*}
    By transitivity of $f_j$, $g$ must be constant.
    This proves the Lemma.
\end{proof}
    We use \cref{Livshitz} to prove \cref{decompprop}. 
    Given $(p_1,...,p_k) \in M$ periodic, we need to show that for each $i \in \{1,...,k\}$,
    \begin{equation*}
        \varphi_{(p_1,...,p_k)}(t_ie_i) =
        \varphi^i_{p_i}(t_i).
    \end{equation*}
    This follows from the definition of $\varphi^i$ and \cref{doesnotdepend}.
 \end{proof}

\section{Proof of \cref{thm1}}\label{sec:thm1}
We will denote by $\varphi$ and $\psi$ the $\mathbb{R}^k$-valued cocycles defining the time changes $\alpha_1$ and $\alpha_2$ respectively.
We are interested in $\varphi$ and $\psi$ up to cohomology.
By \cref{decompprop}, we may write
\begin{equation*}
    \varphi \approx^r \varphi^1 + ... + \varphi^k \ \text{and} \ 
    \psi \approx^r \psi^1 + ... + \psi^k,
\end{equation*}
where $\varphi^i,\psi^i \colon \mathbb{R} \times X_i \to \mathbb{R}^k$ are cocycles over $f_i$ and $\approx^r$ means $C^r$ cohomologous.
\begin{Lemma}
\label{plusminusLemma}
    For $i\in \{1,...,k\}$ and $p_i \in \text{Per}(f_i)$,
    either
    $\mathcal{P}^{\varphi^i}_{p_i} = \mathcal{P}^{\psi^i}_{p_i}$
    or $\mathcal{P}^{\varphi^i}_{p_i} = -\mathcal{P}^{\psi^i}_{p_i}$.
\end{Lemma}
\begin{proof}
    Fix $p_i \in X_i$ be periodic with period $l(p_i)$.
    Let
    \begin{equation*}
        S = \underset{p \in \text{Per}(\alpha_0)}{\sup}
        \max
        \left\{\|\mathcal{P}^{\varphi}_p\|,\|(\mathcal{P}^{\varphi}_p)^{-1}\|^{-1},\|\mathcal{P}^{\psi}_p\|,\|(\mathcal{P}^{\psi}_p)^{-1}\|^{-1}\right\},
    \end{equation*}
    which is finite by \eqref{bddnorm}.
    For $j \in \{1,...,k\}\setminus{i}$, pick $p_j \in X_j$ periodic for $f_j$ with period $l(p_j) > 2S^2|l(p_i)|$.
    Denote by $p = (p_1,...,p_k) \in M$.

    Since $ \mathcal{P}^{\psi}_{p}(l(p_i)e_i) \in \text{Stab}_{\alpha_1}(p)
    =  \mathcal{P}^{\varphi}_{p}(l(p_1)e_1)\mathbb{Z}\oplus...\oplus\mathcal{P}^{\varphi}_{p}(l(p_k)e_k)\mathbb{Z}$,
    we may write
    \begin{equation*}
        \mathcal{P}^{\psi}_{p}(l(p_i)e_i) = 
        \sum_{j} a_j \mathcal{P}^{\varphi}_{p}(l(p_j)e_j), \ a_j \in \mathbb{Z}.
    \end{equation*}
    Therefore,
    \begin{equation*}        ((\mathcal{P}^{\varphi}_p)^{-1}\circ\mathcal{P}^{\psi}_p) (l(p_i)e_i) = \sum_j a_jl(p_j)e_j.
    \end{equation*}
    Note that $\|(\mathcal{P}^{\varphi}_p)^{-1}\circ\mathcal{P}^{\psi}_p) (l(p_i)e_i)\| \leq S^2 |l(p_i)| $.
    Since $l(p_j) > 2S^2l(p_i)$, then $a_j = 0$ for $j\neq i$.
    This means that there exists $a_i \in \mathbb{Z}$ such that $\mathcal{P}^{\psi^i}_{p_i}(l(p_i)) = a_i\mathcal{P}^{\varphi^i}_{p_i}(l(p_i))$.
    Conversely,
    $\mathcal{P}^{\varphi^i}_{p_i}(l(p_i)) = \mathcal{P}^{\psi^i}_{p_i}\left(\frac{l(p_i)}{a_i}\right)$.
    Since $\mathcal{P}^{\varphi^i}_{p_i}(l(p_i)) =\mathcal{P}^{\varphi}_p(l(p_i)e_i) \in \text{Stab}_{\alpha_1}(p) = \text{Stab}_{\alpha_2}(p)$, then
    \begin{equation*}
        (\mathcal{P}^{\psi}_p)^{-1}(\mathcal{P}^{\varphi}_p(l(p_i)e_i)) = \frac{l(p)e_i}{a_i}\in \text{Stab}_{\alpha_0}(p).
    \end{equation*}
    This implies $|a_i| = 1$ and shows the Lemma.    
    \end{proof}
    Fix $\overline{p} = (\overline{p}_1,...,\overline{p}_k) \in M$ periodic.
    In view of the previous Lemma, we may assume after a constant change of coordinates that $\mathcal{P}^{\varphi}_{\overline{p}} = \mathcal{P}^{\psi}_{\overline{p}}$.
    For each $i \in \{1,...,k\}$, let $P_i = \{p_i \in \text{Per}(f_i) \ | \ \mathcal{P}^{\varphi^i}_{p_i} = \mathcal{P}^{\psi^i}_{p_i}\}$
    and $N_i = \{p_i \in \text{Per}(f_i) \ | \ \mathcal{P}^{\varphi^i}_{p_i} = - \mathcal{P}^{\psi^i}_{p_i}\}$.
    By \cref{plusminusLemma}, $P_i \sqcup N_i = \text{Per}(f_i)$. Note that $\overline{p}_i \in P_i$ for each $i$.
    \begin{Lemma}
        $N_i = \emptyset$ for every $i \in \{1,...,k\}$.
    \end{Lemma}
    \begin{proof}
        Suppose by contradiction that $N_i \neq \emptyset$ and 
        let $p_i \in N_i$.
        By \eqref{Periodmap},
        \begin{equation*}
        \mathcal{P}^{\varphi^i}_{p_i}(1) = \int\varphi^i_{x_i}(1)d\mu_{p_i}(x_i) \ \text{and}
        \ \mathcal{P}^{\psi^i}_{p_i}(1) = \int\psi^i_{x_i}(1)d\mu_{p_i}(x_i),
        \end{equation*}
        where $\mu_{p_i}$ is the periodic invariant measure on the orbit of $p_i$.
        Consider the compact and convex set not containing $0$
        \begin{equation*}
            \varphi^i(1) = \left\{
            \int \varphi^i_{x_i}(1)d\mu(x_i) \ \Big| \ 
            \mu \in \mathcal{M}(f_i)
            \right\}.
        \end{equation*}
        Since periodic measures are dense in $\mathcal{M}(f_i)$ by \cite{Sigmund}, then $\{\mathcal{P}^{\varphi^i}_{p_i}(1) \ | \ p_i \in \text{Per}(f_i)\}$ is dense in $\varphi^i(1)$.
        Using that $P_i \sqcup N_i = \text{Per}(f_i)$,
        there exists sequences $p_{i,n} \in P_i$, $q_{i,n} \in N_i$ such that
        \begin{equation*}
            |\mathcal{P}^{\varphi^i}_{p_{i,n}}(1) - \mathcal{P}^{\varphi^i}_{q_{i,n}}(1)| \rightarrow 0.
        \end{equation*}
        Consequently,
        \begin{equation}
        \label{convex0}
            \frac{|\mathcal{P}^{\psi^i}_{p_{i,n}}(1) + \mathcal{P}^{\psi^i}_{q_{i,n}}(1)|}{2}
            = \frac{|\mathcal{P}^{\varphi^i}_{p_{i,n}}(1) - \mathcal{P}^{\varphi^i}_{q_{i,n}}(1)|}{2} 
            \rightarrow 0.
        \end{equation}
        But
        \begin{equation*}
            \psi^i(1) = \left\{
            \int \psi^i_{x_i}(1)d\mu(x_i) \ \Big| \ 
            \mu \in \mathcal{M}(f_i)
            \right\}
        \end{equation*}
        is compact, convex and does not contain $0$. 
        This contradicts \eqref{convex0} and proves the Lemma.
    \end{proof}
     Therefore, up to a constant time change, $\mathcal{P}^{\varphi^i}_{p_i} = \mathcal{P}^{\psi^i}_{p_i}$, for every $p_i \in \text{Per}(f_i)$ and $i \in \{1,...,k\}$.
     For $p = (p_1,...,p_k) \in \text{Per}(\alpha_0)$ and $(a_1,...,a_k) \in \mathbb{R}^k$,
     $\mathcal{P}^{\varphi}_p(a_1,...,a_k) = \mathcal{P}^{\varphi^1}_{p_1}(a_1)+...+\mathcal{P}^{\varphi^k}_{p_k}(a_k)$
     and 
     $\mathcal{P}^{\psi}_p(a_1,...,a_k) = \mathcal{P}^{\psi^1}_{p_1}(a_1)+...+\mathcal{P}^{\psi^k}_{p_k}(a_k)$.
     We have therefore shown $\mathcal{P}^{\varphi}_p = \mathcal{P}^{\psi}_p$, for every $p \in \text{Per}(\alpha_0)$.
     By \cref{Livshitz}, $\varphi$ is $C^r$ cohomologous to $\psi$.
     By \cref{coboundary}, $\alpha_1$ and $\alpha_2$ are $C^r$ conjugate. This finishes the proof of \cref{thm1}.
\section{Proof of \cref{thm2}}\label{sec:thm2}

Let $\varphi$ be the $C^r$ cocycle defining the time change $\alpha$.
By \cref{decompprop}, we may write
\begin{equation*}
    \varphi \approx^r \varphi^1+...\varphi^k,
\end{equation*}
where $\varphi^i \colon \mathbb{R} \times X_i \to \mathbb{R}^k$ are cocycles over $f_i$ and $\approx^r$ means $C^r$ cohomologous.

Since $\mathbb{R}^k = u_1\mathbb{R}\oplus... \oplus u_k\mathbb{R}$,
let $\pi_j \colon \mathbb{R}^k \to u_j\mathbb{R}$ be the projection given by the decomposition.
We also consider $\mathbb{R}^k$ with a norm given by the decomposition so that $\|\pi_j\| = 1$, for all $j \in \{1,...,k\}$.
The hypothesis implies that for all $p \in \text{Per}(\alpha)$ and $a \in \text{Stab}_{\alpha}(p)$, $\pi_j(a) \in \text{Stab}_{\alpha}(p)$ for all $j \in \{1,...,k\}$.

\begin{Lemma}
\label{RuLemma}
    For $i \in \{1,...,k\}$ and $p_i \in X_i$ periodic,
    $\mathcal{P}^{\varphi^i}_{p_i}(\mathbb{R}) \subset u_1\mathbb{R}\cup...\cup u_k\mathbb{R}$.
\end{Lemma}
\begin{proof}
    Fix $p_i \in X_i$ be periodic with period $l(p_i)$.
    Let
    \begin{equation*}
        S = \underset{p \in \text{Per}(\alpha_0)}{\sup}
        \max
        \left\{\|\mathcal{P}^{\varphi}_p\|,\|(\mathcal{P}^{\varphi}_p)^{-1}\|^{-1}\right\},
    \end{equation*}
    which is finite by \eqref{bddnorm}.
    For $j \in \{1,...,k\}\setminus{i}$, pick $p_j \in X_j$ periodic for $f_j$ with period $l(p_j) > 2S^2|l(p_i)|$.
    Denote by $p = (p_1,...,p_k) \in M$.
    
    Since $\mathcal{P}^{\varphi}_p(l(p_i)e_i) \in \text{Stab}_{\alpha}(p)$,
    then $\pi_j \mathcal{P}^{\varphi}_{p}(l(p_i)e_i) \in \text{Stab}_{\alpha}(p)$,
    which implies
    $$(\mathcal{P}^{\varphi}_{p})^{-1}\pi_j\mathcal{P}^{\varphi}_p(l(p_i)e_i) \in \text{Stab}_{\alpha_0}(p) =   l(p_1)e_1\mathbb{Z} \oplus ... \oplus l(p_k)e_k\mathbb{Z}, $$ for every $j \in \{1,...,k\}$.
    However,
    \begin{equation*}
        \|(\mathcal{P}^{\varphi}_{p})^{-1}\pi_j\mathcal{P}^{\varphi}_p(l(p_i)e_i)\|
        \leq S^2|l(p_i)|.
    \end{equation*}
    Since $l(p_l)>2S^2|l(p_l)|$ for $l \in \{1,...,k\}\setminus\{i\}$, 
    we must have $(\mathcal{P}^{\varphi}_{p})^{-1}\pi_j\mathcal{P}^{\varphi}_p(l(p_i)e_i) \in e_i\mathbb{R}$.
    This means $\pi_j\mathcal{P}^{\varphi}_p(l(p_i)e_i) \in \mathcal{P}^{\varphi}_p(e_i\mathbb{R}) \cap u_j\mathbb{R}$, for every $j \in \{1,...,k\}$.
    This is only possible if $\mathcal{P}^{\varphi}_p(e_i\mathbb{R}) = u_j\mathbb{R}$ for some $j \in \{1,...,k\}$.
    This shows the Lemma since $\mathcal{P}^{\varphi^i}_{p_i}(\mathbb{R}) = \mathcal{P}^{\varphi}_p(e_i\mathbb{R})$. 
\end{proof}
Fix $\overline{p} = (\overline{p}_1,...,\overline{p}_k) \in M$ periodic.
We relabel the basis $\{u_1,...,u_k\}$ in order to have $\mathcal{P}^{\varphi^i}_{\overline{p}_i}(\mathbb{R}) = u_i\mathbb{R}$ for every $i \in \{1,...,k\}$.
By applying a constant time change, we may assume that $\mathcal{P}^{\varphi^i}_{\overline{p}_i}(1)$ is a positive multiple of $e_i$ for each $i \in \{1,...,k\}$.
\begin{Lemma}
\label{aligned}
    For $i \in \{1,...,k\}$
    and $p_i \in X_i$ periodic,
    $\mathcal{P}^{\varphi^i}_{p_i}(1)$ is a positive multiple of $e_i$.
\end{Lemma}
\begin{proof}
    Let $p = (\overline{p}_1,...,\overline{p}_{i-1},p_i,\overline{p}_{i+1},...,\overline{p}_k)$.
    For $j \in \{1,...,k\}\setminus \{i\}$,
    $\mathcal{P}^{\varphi}_{p}(e_j) = \mathcal{P}^{\varphi^j}_{\overline{p}_j}(1)$ is a positive multiple of $e_j$.
    But the set $\{\mathcal{P}^{\varphi}_p(e_1),...,\mathcal{P}^{\varphi}_p(e_k)\}$ is linearly independent and contained in $e_1\mathbb{R}\cup ... \cup e_k\mathbb{R}$ by \cref{RuLemma}.
    Therefore, $\mathcal{P}^{\varphi^i}_{p_i}(1) =\mathcal{P}^{\varphi}_p(e_i)   \in e_i\mathbb{R}$.
    By \eqref{Periodmap},
        \begin{equation*}
        \mathcal{P}^{\varphi^i}_{p_i}(1) = \int\varphi^i_{x_i}(1)d\mu_{p_i}(x_i),
        \end{equation*}
        where $\mu_{p_i}$ is the periodic invariant measure on the orbit of $p_i$.
        Consider the compact and convex set not containing $0$
        \begin{equation*}
            \varphi^i(1) = \left\{
            \int \varphi^i_{x_i}(1)d\mu(x_i) \ \Big| \ 
            \mu \in \mathcal{M}(f_i)
            \right\}.
        \end{equation*}
        Since periodic measures are dense in $\mathcal{M}(f_i)$ by \cite{Sigmund}, then $\{\mathcal{P}^{\varphi^i}_{p_i}(1) \ | \ p_i \in \text{Per}(f_i)\}$ is dense in $\varphi^i(1)$.
        But $\{\mathcal{P}^{\varphi^i}_{p_i}(1) \ | \ p_i \in \text{Per}(f_i)\}\subset e_i\mathbb{R}$.
        Therefore, $\varphi^i(1) \subset e_i\mathbb{R}\setminus\{0\}$.
        Since $\mathcal{P}^{\varphi^i}_{\overline{p}_i}(1)$ is a positive multiple of $e_i$, then so is $\mathcal{P}^{\varphi^i}_{p_i}(1)$ for any $p_i \in X_i$ periodic, proving the Lemma.
\end{proof}

For $i \in \{1,...,k\}$, $x_i \in X_i$ and $t_i \in \mathbb{R}$,
consider the $\mathbb{R}$-valued $C^r$ cocycle over $f_i$ given by $\psi^i_{x_i}(t_i) = \pi_i\left(\varphi^i_{x_i}(t_i)\right)$,
where $\pi_i \colon \mathbb{R}^k \to \mathbb{R}$ is the projection on the $i$-th coordinate.
\cref{aligned} implies $\mathcal{P}^{\psi^i}_{p_i}(1)e_i = \mathcal{P}^{\varphi^i}_{p_i}(1)$, for all $p_i \in X_i$ periodic.
Compactness of the set $\varphi^i(1)$ that showed up in the proof of \cref{aligned} implies
\begin{equation*}
    \underline{\psi^i}(1) = \underset{p_i \in \text{Per}(f_i)}{\inf}\psi^i_{p_i}(1) > 0,
\end{equation*}
where we are using the notation introduced in \cref{manesection} and the density of periodic measures for transitive Anosov flows from \cite{Sigmund} in order to consider only periodic measures.
By \cref{softmaneLemma}, there exists a $C^r$ cocycle $\gamma^i$ cohomologous to $\psi^i$ such that $\gamma^i_{x_i}(t_i) > \frac{1}{2}\underline{\psi^i}(1)t_i$.
In particular, $\gamma^i$ is an invertible cocycle, which defines a $C^r$ time change $g_i$ of $f_i$ by \cref{invertiblecocycle}.
Then $\beta = g_1\times...\times g_k$ is a $C^r$ time change of $\alpha_0$ given by the cocycle $\gamma = \sum_i \gamma^ie_i$.
Moreover, for $p=(p_1,...,p_k) \in M$ periodic and $i \in \{1,...,k\}$,
\begin{equation*}
    \mathcal{P}^{\varphi}_p(e_i) = \mathcal{P}^{\psi^i}_{p_i}(1)e_i
    =\mathcal{P}^{\gamma^i}_{p_i}(1)e_i = \mathcal{P}^{\gamma}_p(e_i).
\end{equation*}
By \cref{Livshitz}, $\varphi$ is $C^r$ cohomologous to $\gamma$.
Now \cref{thm2} follows from \cref{coboundary}.

\section{Applications to Cartan actions}
\label{sec:classifyAnosov}

\subsection{Derivative cocycle for Cartan actions}

Let $\alpha \colon \mathbb{R}^k \curvearrowright M$ be a $C^{r}$ Anosov action with $r \geq 1$.
For $\{a_1,...,a_n\} \subset \mathbb{R}^k$ a collection of Anosov elements, we define their common stable manifold at $x \in M$, denoted as $W^s_{a_1,...,a_n}(x)$, to be the connected component at $x$ of $ \cap_{i=1}^nW^s_{a_i}(x)$.
We define a \emph{coarse Lyapunov foliation} to be a common stable foliation of smallest dimension.
Let $\{\mathcal{W}_i\}_{i \in I}$ be the collection of coarse Lyapunov foliations.
Each $\mathcal{W}_i$ is a H\"older foliation with $C^r$ leaves.
Moreover, denoting the associated coarse Lyapunov distribution as $E_i = T\mathcal{W}_i$, the tangent bundle of $M$ admits the $D\alpha$-invariant splitting $TM = T\mathcal{O} \oplus \bigoplus_i E_i$.
A proof of these facts can be found in Lemma 4.5 and Corollary 4.6 of \cite{SpatzierVinhage}.

We say that the action is \emph{Cartan} if all coarse Lyapunov foliations are $1$-dimensional.
For $i \in I$, define the \emph{$i$-th Lyapunov cocycle} to be  $\varphi^i_x(a) = \log \|D_x\alpha(a)\big|_{E_i}\|$.
The Cartan assumption guarantees that $\varphi^i$ is a $\mathbb{R}$-valued cocycle over $\alpha$.

For $\mu$ an ergodic $\alpha$-invariant measure on $M$, there exists a family of functionals $\Delta_{\mu} \subset (\mathbb{R}^k)^*$ called the \emph{Lyapunov functionals} and a measurable $\alpha$-invariant splitting $TM = \bigoplus_{\lambda \in \Delta_\mu}E_{\lambda}$ satisfying for $\lambda \in \Delta_{\mu}$, $a \in \mathbb{R}^k$, $\mu$ almost every $x \in M$ and $v \in E_{\lambda}(x)\setminus\{0\}$,
\begin{equation}
    \label{Lyapunovdef}
    \lim_{n \rightarrow \infty} \frac{1}{n} \log \|D_x\alpha(na)v\| = \lambda(a).
\end{equation}
\begin{Lemma}
    Let $\alpha \colon \mathbb{R}^k \curvearrowright M$ be a $C^r$ Cartan action.
    For $\mu \in \mathcal{M}_{\text{erg}}(\alpha)$,
    $\Delta_{\mu} = \{0\}\cup \{\lambda^i_{\mu}\}_{i \in I}$ with $E_{\lambda^i_{\mu}} = E_i$, for $i \in I$.
    Moreover,
    \begin{equation}
    \label{Lyapint}
        \lambda^i_{\mu}(a) = \int \varphi^i_x(a)d\mu(x).
    \end{equation}
\end{Lemma}
\begin{proof}
    Let $\lambda \in \Delta_{\mu}$ and $A \subset \mathbb{R}^k$ a finite collection of Anosov elements that determine all coarse Lyapunov foliations. We will suppose $-A =A$.
    For $x \in M$ generic, let $v \in E_{\lambda}(x)\setminus\{0\}$ and write $v = v_o + \sum_{i \in I} v_i \in T_x\mathcal{O} \oplus \bigoplus_{i \in I}E_i(x)$.

    If $\lambda = 0$, then $v_i = 0$ for $i \in I$.
    Otherwise, we could pick $a \in A$ that exponentially expands some $v_i$.
    This shows $E_0 = T\mathcal{O}$. For $\lambda \neq 0$, we must have $v_o = 0$ by taking $a \in A$ with $\lambda(a)<0$ since $v_o$ isn't contracted by the action.
    Suppose for $i\neq j \in I$, $v_i \neq 0$ and $v_j \neq 0$. 
    Since they are in different coarse Lyapunov directions, there exists $a \in A$ that exponentially contracts $v_i$ and exponentially expands $v_j$.
    This implies $\lambda(a) >0$.
    But $-a$ exponentially expands $v_i$, which implies $\lambda(-a) >0$ and gives a contradiction.
    Therefore, there exists a unique $i \in I$ such that $v_i  \neq 0$.
    This shows that the splitting induced by the Lyapunov functionals refines the coarse Lyapunov splitting.
    Since we are assuming that the action is Cartan, the splittings must be equal and we may label the Lyapunov functionals as stated in the Lemma.

    For the second statement let $\lambda^i_{\mu} \in \Delta_{\mu}\setminus\{0\}$ and $a \in \mathbb{R}^k$.
    By the cocycle property of $\varphi^i$,
    \begin{equation*}
        \int \varphi^i_x(na) d\mu(x)
        = \int \sum_{i=0}^{n-1} \varphi^i_{\alpha(ia)x}(a)d\mu(x) = n\int \varphi^i_x(a) d\mu(x).
    \end{equation*}
    By \eqref{Lyapunovdef},
    \begin{equation*}
        \int\varphi^i_x(a)d\mu(x)
        = \int \frac{1}{n}\log \|D_x\alpha(na)\big|_{E_i}\|d\mu(x)
        \rightarrow \lambda^i_{\mu}(a).
    \end{equation*}
\end{proof}
For a $\alpha$-invariant measure $\mu \in \mathcal{M}(\alpha)$ not necessarily ergodic, we will still denote $\lambda^i_{\mu}$ to be the functional defined by \eqref{Lyapint}.
Note that it is no longer the classical notion of the Lyapunov exponent for $\mu$, but rather its integral.
Given $i \in I$, let $H^i = \underset{\mu \in \mathcal{M}(\alpha)}{\bigcup}\text{ker}\lambda^i_{\mu}$.

A crucial step in the proof of \cref{thm2} was an averaging argument for real valued cocycles which was done in \cref{manesection}.
We observe that it can also be used to give a characterization of the set of Anosov elements for Cartan actions in terms of its Lyapunov functionals.

\begin{Lemma}
\label{lyapexpthm}
    Let $M$ be a compact, connected smooth manifold and $\alpha \colon \mathbb{R}^k \curvearrowright M$ be a  $C^s$ Cartan action for $s\geq 1$.
    In the above notation, the set of Anosov elements is $\mathbb{R}^k \setminus \bigcup_{i \in I} H^i$.
    Moreover, $\alpha$ is totally Cartan if, and only if, for each $i \in I$, $\ker \lambda^i_{\mu}$ does not depend on the invariant measure $\mu$. 
\end{Lemma}

A statement similar to \cref{lyapexpthm} for rank $1$ actions has already appeared as one of the main results in \cite{Cao} and as Corollary 3.6 of \cite{RHJMD}.
\cref{lyapexpthm} is instead stated for higher rank actions.
Additionally, the Cartan assumption allows for a full characterization of the set of Anosov elements in terms of the Lyapunov functionals.

\begin{proof}
Let $a \in \bigcup_i H^i$, that is, $a \in \text{ker} \lambda^i_{\mu}$ for some $\mu \in \mathcal{M}(\alpha)$ and $i \in I$.
Taking the ergodic decomposition of $\mu$,
\begin{equation*}
    \int\int\varphi^i_x(a)d\nu_y(x)d\mu(y) = 0.
\end{equation*}
This implies either that there exists an ergodic measure $\nu$ satisfying $\int \varphi^i_x(a)d\nu(x) =\lambda^i_{\nu}(a) = 0$ or there exists ergodic measures $\nu_1,\nu_2$ such that $\int \varphi^i_x(a)d\nu_1(x) = \lambda^i_{\nu_1}(a) <0$ and $\int \varphi^i_{x}(a) d\nu_2(x) = \lambda^i_{\nu_2}(a) > 0$.
In the first case, we see that $a$ is not Anosov since for $x \in M$ $\nu$, $E_i(x)$ is not exponentially contracted or expanded.
For the second case, $a$ is also not Anosov.
To see this, take $x \in M$ $\nu_1$ generic and $y \in M$ $\nu_2$ generic. Then $E_i(x)$ must exponentially contract, while $E_i(y)$ exponentially expands, which is impossible for an Anosov element by continuity of $E_i$ and connectedness of $M$.

Now let $a \in \mathbb{R}^k \setminus \bigcup_i H^i$.
For $i \in I$ and $\mu' \in \mathcal{M}(\alpha)$, suppose without loss of generality that $\lambda^i_{\mu'}(a)>0$.
We first claim that $\underline{\varphi^i}(a) = \underset{\mu \in \mathcal{M}(\alpha)}{\inf}\lambda^i_{\mu}(a)>0$.
Otherwise, by continuity of $\mu \mapsto \lambda^i_{\mu}(a)$ and convexity of $\mathcal{M}(\alpha)$, there would exist $\mu \in \mathcal{M}(\alpha)$ with $\lambda^i_{\mu}(a) = 0$.
By \cref{softmaneLemma}, there exists $\varepsilon_i>0$ and $\psi^i$ cocycle cohomologous to $\varphi^i$ satisfying $\psi^i_x(ta) >\varepsilon_i t$, for all $t >0$ and $x \in M$.
Let $H \colon M \to \mathbb{R}$ be such that
\begin{equation*}
    \varphi^i_x(b) =\psi^i_x(b)+H(\alpha(b)x)-H(x).
\end{equation*}
Let $\|\cdot\|$ be the standard metric on $M$ Define a new metric $|\cdot|_i$ on $E_i(x)$ by $|v|_i = e^{H(x)}\|v\|$.
For $x \in M$, $t>0$ and $v \in E_i(x)$,
\begin{equation*}
    \left|D_x\alpha(ta)\left(\frac{v}{|v|_i}\right)\right|_i
    = e^{H(\alpha(ta)x)-H(x)}\left\|D_x\alpha(ta)\left(\frac{v}{\|v\|}\right)\right\|
    = e^{\psi^i_x(ta)} > e^{\varepsilon_it}.
\end{equation*}
This shows that $E_i$ uniformly expands under $ta$.
Repeating the construction for each $i \in I$ shows that $a$ is an Anosov element and finishes the proof.
\end{proof}

\subsection{$C^1$ time changes of Cartan actions}

Let $\alpha\colon\R^k\curvearrowright M$ be a $C^s$ Cartan action for $s\geq1$. Let $TM=T\mathcal{O}\oplus(\bigoplus_{i\in I}E_i^\alpha)$ be the decomposition into coarse Lyapunov subbundles. 
Recall that, for a periodic point $p$, the $i$-th Lyapunov functional $\lambda^{i,\alpha}_p$ of the action $\alpha$ is defined by $\lambda^{i,\alpha}_p(a)=\int\operatorname{log}\|D_x\alpha(a)|_{E^\alpha_i}\|d\mu_p(x)$, where $\mu_p$ is the ergodic invariant probability measure supported on the periodic orbit of $p$. In this subsection, $\mathcal{P}^\varphi_p\colon\R^k\rightarrow\R^k$ is the linear map extending $\varphi_p\colon\operatorname{Stab}_\beta(p)\rightarrow\R^k$.

First, we observe how the invariant splitting changes at each periodic point under the $C^1$ time change action. 
\begin{Lemma}
\label{changeofcoord}
    Let $\beta$ be a $C^1$ time change of $\alpha$ via an invertible $C^1$ cocycle $\varphi$. 
    For $p\in\operatorname{Per}(\beta)$ and $a\in \operatorname{Stab}_\beta(p)$, we have the $D_p\beta(a)$ invariant splitting $$T_pM=T_p\mathcal{O}\oplus\bigoplus_{i\in I}E^\beta_i(p).$$
\end{Lemma}
\begin{proof}
    We define the derivatives of the actions $\alpha$ and $\beta$ with respect to the manifold variable by
    \begin{align*}
         A(a):=D^M_{(\varphi_p(a),p)}\alpha\colon T_pM\rightarrow T_pM, \quad
         B(a):=D^M_{(a,p)}\beta\colon T_pM\rightarrow T_pM.
    \end{align*}
    By the chain rule, for $v\in T_pM,$ $$B(a)v=D^{\R^k}_{(\varphi_p(a),p)}\alpha(D^M_{(a,p)}\varphi(v))+D^M_{(\varphi_p(a),p)}\alpha(v).$$
    Since the derivative with respect to the $\R^k$ parameter $D^{\R^k}_{(\varphi_p(a),p)}\alpha$ takes values in the tangent space to the orbit $T_p\mathcal{O},$ it follows that $$(B(a)-A(a))(T_pM)\subset T_p\mathcal{O}.$$
    For each $i\in I$, choose $a_i\in \operatorname{Stab}_\beta(p)$ such that $\lambda^i_p(\mathcal{P}^\varphi_p(a_i))<0.$ Define 
    \begin{align*}
        \sigma^i\colon E^\alpha_i(p)\rightarrow T_p\mathcal{O}, \quad \quad v\mapsto \sum_{n\geq0}(A(a_i)-B(a_i))(A(a_i))^nv.
    \end{align*}
    This is well defined since $(A(a_i))^nv=e^{\lambda^i_p(\mathcal{P}^\varphi_p(a_i))n}v$ converges exponentially to $0$. Next, define
    \begin{align*}
        E^\beta_i(p)&:=\{v+\sigma^i(v)\colon v\in E^\alpha_i(p)\} =\operatorname{graph}(\sigma^i)\subset E^\alpha_i(p)\oplus T_p\mathcal{O}.
    \end{align*}
    We claim that for every $b\in \operatorname{Stab}_\beta(p)$, $$B(b)(E^\beta_i(p))=E^\beta_i(p).$$ 
    Indeed, for $v\in E^\alpha_i(p),$ $$B(b)(v+\sigma^i(v))=A(b)v+(B(b)-A(b))v+B(b)(\sigma^i(v)).$$ Here, $A(b)v\in E^\alpha_i(p),$ while the remaining terms lie in $T_p\mathcal{O}$, since $$B(b)|_{T_p\mathcal{O}}=\operatorname{id}_{T_p\mathcal{O}}.$$ Thus, the claim follows once we verify that
    $$\sigma^i(A(b)v)=\sigma^i(v)-(A(b)-B(b))v.$$
    For $a\in \operatorname{Stab}_\beta(p),$ define $$C(a):=(A(a)-B(a))|_{E^\alpha_i(p)}.$$ Then
    \begin{align*}
        C(a+b)&=A(a)A(b)-B(a)B(b) \\
        &=(A(a)-B(a))A(b)+B(a)(A(b)-B(b)) \\
        &=C(a)e^{\lambda^i_p(\mathcal{P}^\varphi_p(b))}+C(b) \\
        &=C(b)e^{\lambda^i_p(\mathcal{P}^\varphi_p(a))}+C(a).
    \end{align*}
    The last equality follows from $a+b=b+a.$ Using this notation, we now compare $\sigma^i(A(b)v)$ and $\sigma^i(v)-C(b)v$.
    \begin{align*}
        \sigma^i(A(b)v)&=\sum_{n\geq0}(C(a)e^{\lambda^i_p(\mathcal{P}^\varphi_p(b))}(A(a))^nv \\
        &=\sum_{n\geq0}(C(a+b)-C(b))(A(a))^nv,    \end{align*}
    \begin{align*}
        \sigma^i(v)-C(b)v&=\sum_{n\geq0}(C(a+b)-C(b)e^{\lambda^i_p(\mathcal{P}^\varphi_p(a))})(A(a))^nv-C(b)v \\
        &=\sum_{n\geq0}C(a+b)(A(a))^nv-\sum_{n\geq1}C(b)(A(a))^nv-C(b)v \\
        &=\sum_{n\geq0}(C(a+b)-C(b))(A(a))^nv.
    \end{align*}
    Thus, $\sigma^i(A(b)v)=\sigma^i(v)-(A(b)-B(b))v,$ which completes the proof.
\end{proof}
Next, we describe how the Lyapunov functional changes at each periodic point under $C^1$ time changes. 

\begin{Lemma}
\label{changeofcoordLyap}
    Let $\beta$ be a $C^1$ time change of $\alpha$ via an invertible $C^1$ cocycle $\varphi$. 
    Let $p\in\operatorname{Per}(\alpha)$. Then, for each $i\in I$,  $$\lambda^{i,\beta}_p=\lambda^{i,\alpha}_p\circ\mathcal{P}^\varphi_p.$$
\end{Lemma}
\begin{proof}
    Since $\beta(a)p=\alpha(\varphi_p(a))p=p$ for $p\in \operatorname{Per}(\alpha)$ and $a\in\operatorname{Stab}_\beta(p)$, we have $\varphi_p(a)\in\operatorname{Stab}_\alpha(p).$ By \cref{changeofcoord}, there is $D_p\beta(a)$ invariant splitting $$T_pM=T_p\mathcal{O}\oplus\bigoplus_{i\in I}E_i^\beta(p),$$ and each $E^\beta_i(p)$ projects isomorphically onto $E^\alpha_i(p)$ along the orbit direction $T_p\mathcal{O}.$ For each $i\in I,$ there exists a linear map $\sigma^i(p)\colon E^\alpha_i(p)\rightarrow T_p\mathcal{O}$ such that $$E^\beta_i(p)=\{v+\sigma^i(p)v\colon v\in E^\alpha_i(p)\}.$$

    By the chain rule,  
    $$D^M_{(a,p)}\beta(v+\sigma^i(p)v)=D^{\R^k}_{(\varphi_p(a),p)}\alpha(D^M_{(a,p)}\varphi(v+\sigma^i(p)v))+D^M_{(\varphi_p(a),p)}\alpha( v+\sigma^i(p)v).$$

    Decomposing the differential of $\alpha$ into its $\mathbb{R}^k$ and $M$ components, we obtain
    \begin{align*}
    D^{\R^k}_{(\varphi_p(a),p)}\alpha(D^M_{(a,p)}\varphi(v+\sigma^i(p)v)+D^M_{(\varphi_p(a),p)}\alpha(v)+D^M_{(\varphi_p(a),p)}\alpha(\sigma^i(p)v).
    \end{align*}

    The first term $D^{\R^k}_{\varphi((a,x),x)}\alpha(\cdot)$ is tangent to the $\alpha$-orbit since differentiation in the $\R^k$-variable produces a vector in $T_p\mathcal{O}$. The third term is also tangent to the orbit because $\sigma^i(p)v\in T_p\mathcal{O}$ and $T_p\mathcal{O}$ is $D^M_{(\varphi_p(a),p)}\alpha$ invariant. Therefore, modulo $T_p\mathcal{O},$ we have
    $$D^M_{(a,p)}\beta(v+\sigma^i(p)v)=D^M_{(\varphi_p(a),p)}\alpha(v).$$
    
    Thus, under  the natural projection $\pi_i\colon E^\beta_i(p)\rightarrow E^\alpha_i(p)$ along $T_p\mathcal{O},$ $D_p\beta(a)|_{E^\beta_i(p)}$ is conjugate to $D_p\alpha(\varphi_p(a))|_{E^\alpha_i(p)}$. In particular, $$\operatorname{log}\|D_p\beta(a)|_{E^\beta_i(p)}\|=\operatorname{log}\|D_p\alpha(\varphi_p(a))|_{E^\alpha_i(p)}\|.$$ Equivalently, $\lambda^{i,\beta}_p(a)=\lambda^{i,\alpha}_p(\varphi_p(a))$ and it follows $$\lambda^{i,\beta}_p(a)=\lambda^{i,\alpha}_p(\mathcal{P}^\varphi_p(a))$$ since $\varphi_p(a)=\mathcal{P}^\varphi_p(a).$ As both sides are linear on $\operatorname{Stab}_\beta(p),$ we conclude that $$\lambda^{i,\beta}_p=\lambda^{i,\alpha}_p\circ\mathcal{P}^\varphi_p.$$ This completes the proof.
\end{proof}

\subsection{Proof of \cref{thm4}}

We are now ready to prove \cref{thm4}.

\subsubsection*{$(1) \implies (2)$}

Each $g_i$ is Anosov since it is a $C^1$ time change of an Anosov flow.
Therefore, $g_1 \times ... \times g_k$ is totally Anosov.
Since $C^1$ conjugation up to automorphisms preserves the totally Anosov property, $\alpha$ is totally Anosov.

\subsubsection*{$(2) \implies (3)$}

Suppose that $\alpha$ is totally Anosov.
For $i \in I$, we will show that $\ker(\lambda^{i,\alpha}_p)$ does not depend on $p \in \text{Per}(\alpha)$.

Since $\alpha_0$ is the product action of Anosov flows on $3$-dimensional manifolds, then the coarse Lyapunov splitting is given by
\begin{equation*}
    TM = T\mathcal{O}_{\alpha_0} \oplus \bigoplus_{i = 1}^k E^s_{f_i} \oplus \bigoplus_{i=1}^k E^u_{f_i}.
\end{equation*}
We index the coarse Lyapunov distributions by $I = \{\pm 1,...,\pm k\}$, where $i$ corresponds to $E^u_{f_i}$ while $-i$ corresponds to $E^s_{f_i}$.
Moreover, denoting by $\{e_1,...,e_k\}$ the standard basis of $\mathbb{R}^k$,
\begin{equation*}
    \ker(\lambda^{i,\alpha_0}_p) = \text{span}\{e_j \ | \ j \neq |i|\},
\end{equation*}
for any $p \in \text{Per}(\alpha_0)$.

Let $\varphi \colon \mathbb{R}^k \times M \to \mathbb{R}^k$ be the $C^1$ cocycle over $\alpha_0$ defining the action $\alpha$, that is, $\alpha_0(x) = \alpha(\varphi(a,x))x$, for all $a \in \mathbb{R}^k$ and $x \in M$.
By \cref{changeofcoord},
\begin{equation}
\label{eq:ker}
\ker(\lambda_p^{i,\alpha}) = \text{span}\{\mathcal{P}^{\varphi}_p(e_j) \ | j \neq |i| \}, \ \text{for} \ p \in \text{Per}(\alpha).
\end{equation}
 In particular, 
 $\mathcal{P}^{\varphi}_p(e_j)$ is not an Anosov element for $\alpha$, for any $j \in \{1,...,k\}$.

By \cref{decompprop}, we may write $\varphi \approx^1  \varphi^1 + ...+ \varphi^k$, where $\varphi^j$ is a cocycle over the flow $f_j$. Consequently, $\mathcal{P}^{\varphi^j}_{p_j}(1) = \mathcal{P}^{\varphi}_p(e_j)$ is not an Anosov element of $\alpha$ for any $p = (p_1,...,p_k) \in \text{Per}(\alpha_0)$ and $j \in \{1,...,k\}$.

Fix $i \in I$, $j \in \{1,...,k\}\setminus \{|i|\}$
and $\overline{p}_l \in \text{Per}(f_l)$ for $l \notin \{j,|i|\}$.
We claim that $H(p_j) = \ker\left(\lambda^{i,\alpha}_{(\overline{p}_1,...,\overline{p}_{j-1},p_j,\overline{p}_{j+1},...,\overline{p}_k)}\right)$ does not depend on $p_j \in \text{Per}(f_j)$.
Indeed, denoting by $V = \text{span}\{\mathcal{P}^{f_l}_{\overline{p}_l}(1) \ | \ l \in \{1,...,k\} \setminus \{j,|i|\}\}$, we obtain by \cref{eq:ker}
\begin{equation*}
    H(p_j) = \langle \mathcal{P}^{f_j}_{p_j}(1) \rangle \oplus V.
\end{equation*}
If $H$ is not constant, there exists $p_j,q_j \in \text{Per}(f_j)$ such that
\begin{equation*}
    \langle \mathcal{P}^{f_j}_{p_j}(1) \rangle \oplus \langle \mathcal{P}^{f_j}_{q_j}(1)\rangle \oplus V = \mathbb{R}^k.
\end{equation*}
However, since periodic measures for $f_j$ are weak* dense in the set of all invariant measures $\mathcal{M}(f_j)$,
\begin{equation}
\label{eq:line}
    \overline{\{\mathcal{P}^{\varphi^j}_p(1) \ | \ p \in \text{Per}(f_j)\}} =
    \{ \mathcal{P}^{\varphi^j}_{\mu}(1) \ | \ \mu \in \mathcal{M}(f_j)\}.
\end{equation}
Since $\mathcal{M}(f_j)$ is compact and convex and the map $\mu \mapsto \mathcal{P}^{f_j}_{\mu}(1)$ is linear and continuous, the right hand side of \eqref{eq:line} is a compact and convex set of $\mathbb{R}^k$.
In particular, $H(\mathcal{M}(f_j)) = \bigcup_{\mu \in \mathcal{M}(f_j)} \langle \mathcal{P}^{f_j}_{\mu}(1)\rangle \oplus V $ contains  $\bigcup_{u \in [\mathcal{P}^{f_j}_{p_j}(1),\mathcal{P}^{f_j}_{q_j}(1)]} \langle u \rangle \oplus V$ which contains an open set of $\mathbb{R}^k$.
However, $\bigcup_{p \in \text{Per}(f_j)}\langle \mathcal{P}^{f_j}_p(1)\rangle \oplus V$ is a set of non Anosov elements which is dense in $H(\mathcal{M}(f_j))$. This contradicts totally Anosov and proves the claim that $p \mapsto H(p)$ is constant.

Since the coordinate fixed was arbitrary, this shows $\ker(\lambda^{i,\alpha}_p)$ does not depend on $p \in \text{Per}(\alpha)$, which shows $(3)$.

\subsubsection*{$(3)\implies (1)$}

For $i \in I$, let $H^i = \ker \lambda^{i,\alpha}_p$ which does not depend on $p$.
Let $\varphi$ be the cocycle over $\alpha_0$ defining the time change $\alpha$.
By \cref{changeofcoordLyap}, $\lambda^{i,\alpha_0}_p = \lambda^{i,\alpha}_p \circ \mathcal{P}^{\varphi}_p$.
Therefore, 
\begin{equation*}
    \mathcal{P}^{\varphi}_p(\langle e_1,...,\hat{e_i},...,e_k\rangle) = H^i.
\end{equation*}
In particular,
\begin{equation*}
    \mathcal{P}^{\varphi}_p(\langle e_i\rangle) = \bigcap_{j \neq i}H^j = \langle u_i \rangle,
\end{equation*}
for some $u_i \in \mathbb{R}^k$.
Since for all $p \in \text{Per}(\alpha_0)$ and $i \in I$, there exists $a_i(p)\in \mathbb{R}$ such that $\text{Stab}_{\alpha_0}(p) = a_1(p)e_1\mathbb{Z} \oplus ... \oplus a_k(p)e_k \mathbb{Z}$ and $\text{Stab}_{\alpha}(p) = \mathcal{P}^{\varphi}_p(\text{Stab}_{\alpha_0}(p))$,
then there exists $b_i(p) \in \mathbb{R}$ satisfying $\text{Stab}_{\alpha}(p) = b_1(p)u_1\mathbb{Z} \oplus ... \oplus b_k(p)u_k \mathbb{Z}$.
Therefore, the hypothesis for \cref{thm2} are satisfied and we obtain (1).

\subsection{More counterexamples for the Katok Spatzier conjecture}

Here we prove \cref{counterexamples}.
The proof closely follows the main argument in \cite{Vinhage} up until showing the absence of rank $1$ factors.

Let $X,Y$ be compact, connected, Riemannian manifolds of dimension $3$ and $f^s \colon X \to X$ and $g^t \colon Y \to Y$ be transitive Anosov flows.
Denote by $\alpha_0 = f^s \times g^t$.
Let $(p_1,q_1) \neq (p_2,q_2) \in \text{Per}(\alpha_0)$ and for $\delta>0$, define $u_{\delta} \in C^{\infty}(X), v_{\delta} \in C^{\infty}(Y)$ such that
\begin{enumerate}
    \item $u_{\delta}(f^s(p_1)) = \frac{\delta}{l(p_1)}$, $v_{\delta}(g^t(q_1)) = \frac{\delta}{l(q_1)}$, for all $s, t \in \mathbb{R}$;
    \item $u_{\delta}(f^s(p_2)) = v_{\delta}(g^t(q_2)) = 0$, for all $s,t \in \mathbb{R}$;
    \item $|u_{\delta}|,|v_{\delta}| < \max \left\{\frac{\delta}{l(p_1)},\frac{\delta}{l(q_1)}\right\}$.
\end{enumerate}
Consider the $\mathbb{R}^2$ valued cocycle over $\alpha_0$
\begin{equation*}
    \varphi^{\delta}_{(x,y)}(s,t) = \left( s- \int_0^t v_{\delta}(g^{\tau}(y))d\tau, t+ \int_0^s u_{\delta}(f^{\tau}(x))d\tau \right).
\end{equation*}
For $\delta>0$ small enough, $\varphi^{\delta}$ satisfies the conditions of Lemma 2.4 of \cite{Vinhage}.
Therefore, it is an invertible cocycle and determines a $C^1$ time change of $\alpha_0$ which we denote by $\alpha_{\delta}$.

\begin{Lemma}
    For $\delta >0$ enough, $\alpha_{\delta}$ is Cartan, but not totally Cartan.
    In particular, $\alpha_{\delta}$ is not homogeneous.
\end{Lemma}
\begin{proof}
    Just as in the proof of Theorem 5.1 of \cite{Vinhage}, $(\pm1,\pm1)$ are still Anosov elements if $\delta>0$ is small enough with the same dimensions of stable and unstable manifolds.

    A straightforward computation shows
    \begin{equation*}
    \mathcal{P}^{\varphi^{\delta}}_{(p_1,q_1)}=
        \begin{pmatrix}
            1 & -\delta \\
            \delta & 1
        \end{pmatrix} \ \text{and} \ 
        \mathcal{P}^{\varphi^{\delta}}_{(p_2,q_2)} = 
        \begin{pmatrix}
            1 & 0 \\
            0 & 1
        \end{pmatrix}.
    \end{equation*}
    By \cref{changeofcoordLyap}, $\lambda^{i,\alpha_{\delta}}_{(p_1,q_1)}$ is not a multiple of $\lambda^{i,\alpha_{\delta}}_{(p_2,q_2)}$.
    By \cref{thm4}, this implies that $\alpha_{\delta}$ is not totally Cartan.

    For homogeneous actions, the growth of vectors by the derivative is determined by the adjoint representation.
    If $\alpha_{\delta}$ were a Cartan homogeneous action, its one dimensional coarse Lyapunov bundles would be spanned by joint eigenvectors of the $\mathbb{R}^k$ action by the adjoint representation.
    Each bundle would come with a functional and the set of Anosov elements would be the complement of the kernels of such kernels. In particular, $\alpha_{\delta}$ would be totally Anosov, which it is not. This concludes the proof of the Lemma.
\end{proof}

    We now show that, for many $\delta >0$, $\alpha_{\delta}$ has no rank one factors.
    But first, we show a useful Lemma.
    We say that two lattices $\Gamma_1,\Gamma_2 < \mathbb{R}^2$ \emph{share a line} if there exists $L < \mathbb{R}$ one dimensional subspace such that $L \cap \Gamma_1 \neq \{0\} \neq L \cap \Gamma_2$.
    \begin{Lemma}
    \label{stabrot}
        Given two lattices $\Gamma_1,\Gamma_2 < \mathbb{R}^2$, there exists at most countably many $\delta$ such that $A_{\delta}\Gamma_1$ and $\Gamma_2$ share a line, where 
        \begin{equation*}
            A_{\delta}=
        \begin{pmatrix}
            1 & -\delta \\
            \delta & 1
        \end{pmatrix}.
        \end{equation*}
    \end{Lemma}
    \begin{proof}
        Let $p \colon \mathbb{R}^2\setminus \{0\} \to S^1$ be the usual projection map onto the circle.
        It is clear that two lattices $\Gamma_1',\Gamma_2' < \mathbb{R}^2$ share a line if, and only if, $p(\Gamma_1'\setminus\{0\})\cap p(\Gamma_2'\setminus \{0\}) \neq \emptyset$.
        Denote by $\Lambda_i = p(\Gamma_i \setminus \{0\})$, for $i = 1,2$.
        For each $\delta\geq0$, there exists a unique $\theta(\delta) \in [0,\frac{\pi}{2})$ satisfying $p(A_{\delta}(\Gamma_1\setminus \{0\})) = R_{\theta(\delta)}\Lambda_1$, where $R_{\theta}$ is rotation by $\theta$ on $S^{1}$.

        We wish to show that for at most countably many $\delta >0$, $R_{\theta(\delta)}\Lambda_1 \cap \Lambda_2 \neq \emptyset$.
        But $R_{\theta}\Lambda_1 \cap\Lambda_2 \neq \emptyset$ if, and only if, $\theta \in \Lambda_2-\Lambda_1$, where we think of $S^1 = \mathbb{R}/2\pi\mathbb{Z}$.
        Since $\Lambda_1,\Lambda_2$ are countable, then $\Lambda_2-\Lambda_1$ is countable.
        Therefore, only a countable set of $\delta$ can satisfy $\theta(\delta) \in \Lambda_2 - \Lambda_1$, which shows the Lemma.
    \end{proof}

    \begin{Lemma}
        Aside from possibly countably many $\delta >0$, $\alpha_{\delta}$ has no $C^1$ rank one factors.
    \end{Lemma}

    \begin{proof}
        Let $h^t \colon Z \to Z$ be a nontrivial $C^1$ rank one factor of $\alpha_{\delta}$ with corresponding homomorphism $\sigma \colon \mathbb{R}^2 \to \mathbb{R}$ and $C^1$ projection map $\pi \colon M \to Z$.
        Suppose there exists $i \in I$ and $x \in M$ such that $E^{\alpha_{\delta}}_i(x) \cap \ker D_x\pi =\{0\}$.
        By Lemma 4.3 of \cite{Vinhage}, there exists a continous metric on $E_i^{\alpha_{\delta}}$ such that for all $a \in \ker \sigma$, $D\alpha(a)|_{E_i^{\alpha_{\delta}}}$ is an isometry.
        This implies that $\ker \lambda^{i,\alpha_{\delta}}_p = \ker \sigma$, for all $p \in \text{Per}(\alpha)$.
        However, this is impossible as $\ker \lambda^{i,\alpha_{\delta}}_{(p_1,q_1)}$ and $\ker \lambda^{i,\alpha_{\delta}}_{(p_2,q_2)}$ are not proportional.
        Therefore, $\bigoplus_{i \in I} E^{i,\alpha_{\delta}} \subset \ker D\pi$.
        By Lemma 4.1 of \cite{Vinhage}, $h^t$ must be a transitive circle rotation.

        In particular, for each $p \in \text{Per}(\alpha_{\delta})$, $\sigma (\text{Stab}_{\alpha_{\delta}}(p)) \subset \text{Stab}_{h^t}(\pi(p))$ which is a lattice in $\mathbb{R}$.
        This implies that $\ker \sigma \cap \text{Stab}_{\alpha_{\delta}}(p) \neq \{0\}$, that is, there exists a line $L \subset \mathbb{R}^2$ passing through the origin such that $L \cap \text{Stab}_{\alpha_{\delta}}(p) \neq \{0\}$, for all $p \in \text{Per}(\alpha_{\delta})$.
        In particular, $\text{Stab}_{\alpha_{\delta}}(p_1,q_1)$ and $\text{Stab}_{\alpha_{\delta}}(p_2,q_2)$ share a line.
        Following the notation of \cref{stabrot}, $\text{Stab}_{\alpha_{\delta}}(p_1,q_1) = A_{\delta}\text{Stab}_{\alpha_0}(p_1,q_1)$ and $\text{Stab}_{\alpha_{\delta}}(p_2,q_2) = \text{Stab}_{\alpha_0}(p_1,q_1)$.
        By \cref{stabrot}, this is only possible for at most countably many $\delta>0$,
        which proves the Lemma.
    \end{proof}

\section*{Acknowledgements}
We are grateful to Kurt Vinhage for proposing this problem, for the fruitful conversations and for the comments on an earlier draft of this paper.
The first and second author would like to thank their respective PhD advisors Aaron Brown and David Fisher for their guidance throughout this project.
The authors would also like to thank Homin Lee and Ralf Spatzier for the conversations had and James Marshall Reber for comments on an earlier draft of this paper.

This research was supported in part by grants from the NSF (DMS-2235451) and Simons Foundation (MPS-NITMB-00005320) to the NSF-Simons National Institute for Theory and Mathematics in Biology (NITMB).
This material is also based upon work supported by the National Science Foundation under Grant
No. DMS-2020013 and DMS-2400191.

\bibliography{references}

@article {KatokSpatzier,
    AUTHOR = {Katok, A. and Spatzier, R. J.},
     TITLE = {First cohomology of {A}nosov actions of higher rank abelian
              groups and applications to rigidity},
   JOURNAL = {Inst. Hautes \'Etudes Sci. Publ. Math.},
  FJOURNAL = {Institut des Hautes \'Etudes Scientifiques. Publications
              Math\'ematiques},
    NUMBER = {79},
      YEAR = {1994},
     PAGES = {131--156},
      ISSN = {0073-8301,1618-1913},
   MRCLASS = {58F15 (22E40)},
  MRNUMBER = {1307298},
MRREVIEWER = {Alexander\ Starkov},
       URL = {http://www.numdam.org/item?id=PMIHES_1994__79__131_0},
}

@article {Vinhage,
    AUTHOR = {Vinhage, K.},
     TITLE = {Instability for rank one factors of product actions},
   JOURNAL = {J. Mod. Dyn.},
  FJOURNAL = {Journal of Modern Dynamics},
    VOLUME = {21},
      YEAR = {2025},
     PAGES = {607--620},
      ISSN = {1930-5311,1930-532X},
   MRCLASS = {37C85 (37A17 37C15 37C20 37C79 37D20 37D40)},
  MRNUMBER = {4973403},
       DOI = {10.3934/jmd.2025013},
       URL = {https://doi.org/10.3934/jmd.2025013},
}

@article {LopesThieullen,
    AUTHOR = {Lopes, A. O. and Thieullen, Ph.},
     TITLE = {Sub-actions for {A}nosov flows},
   JOURNAL = {Ergodic Theory Dynam. Systems},
  FJOURNAL = {Ergodic Theory and Dynamical Systems},
    VOLUME = {25},
      YEAR = {2005},
    NUMBER = {2},
     PAGES = {605--628},
      ISSN = {0143-3857,1469-4417},
   MRCLASS = {37D20},
  MRNUMBER = {2129112},
MRREVIEWER = {Oliver\ Jenkinson},
       DOI = {10.1017/S0143385704000732},
       URL = {https://doi.org/10.1017/S0143385704000732},
}

@article {Sigmund,
    AUTHOR = {Sigmund, K.},
     TITLE = {On the space of invariant measures for hyperbolic flows},
   JOURNAL = {Amer. J. Math.},
  FJOURNAL = {American Journal of Mathematics},
    VOLUME = {94},
      YEAR = {1972},
     PAGES = {31--37},
      ISSN = {0002-9327,1080-6377},
   MRCLASS = {28A65 (54H20)},
  MRNUMBER = {302866},
MRREVIEWER = {Robert\ Ellis},
       DOI = {10.2307/2373591},
       URL = {https://doi.org/10.2307/2373591},
}

@article {SpatzierVinhage,
    AUTHOR = {Spatzier, R. and Vinhage, K.},
     TITLE = {Cartan actions of higher rank abelian groups and their
              classification},
   JOURNAL = {J. Amer. Math. Soc.},
  FJOURNAL = {Journal of the American Mathematical Society},
    VOLUME = {37},
      YEAR = {2024},
    NUMBER = {3},
     PAGES = {731--859},
      ISSN = {0894-0347,1088-6834},
   MRCLASS = {37C85 (37C15 37C40 37D20)},
  MRNUMBER = {4736528},
MRREVIEWER = {Sanghoon\ Kwon},
       DOI = {10.1090/jams/1033},
       URL = {https://doi.org/10.1090/jams/1033},
}

@incollection {LTDiffeos,
    AUTHOR = {Lopes, A. O. and Thieullen, Philippe},
     TITLE = {Sub-actions for {A}nosov diffeomorphisms},
      NOTE = {Geometric methods in dynamics. II},
   JOURNAL = {Ast\'erisque},
  FJOURNAL = {Ast\'erisque},
    NUMBER = {287},
      YEAR = {2003},
     PAGES = {xix, 135--146},
      ISSN = {0303-1179,2492-5926},
   MRCLASS = {37D20 (37A20)},
  MRNUMBER = {2040005},
MRREVIEWER = {Viorel\ Ni\c tic\u a},
}

@article {KatokSpatziermeasure,
    AUTHOR = {Katok, A. and Spatzier, R. J.},
     TITLE = {Invariant measures for higher-rank hyperbolic abelian actions},
   JOURNAL = {Ergodic Theory Dynam. Systems},
  FJOURNAL = {Ergodic Theory and Dynamical Systems},
    VOLUME = {16},
      YEAR = {1996},
    NUMBER = {4},
     PAGES = {751--778},
      ISSN = {0143-3857,1469-4417},
   MRCLASS = {58F11},
  MRNUMBER = {1406432},
MRREVIEWER = {Scot\ Adams},
       DOI = {10.1017/S0143385700009081},
       URL = {https://doi.org/10.1017/S0143385700009081},
}

@article {KSsmoothrigid,
    AUTHOR = {Katok, A. and Spatzier, R. J.},
     TITLE = {Differential rigidity of {A}nosov actions of higher rank
              abelian groups and algebraic lattice actions},
   JOURNAL = {Tr. Mat. Inst. Steklova},
  FJOURNAL = {Trudy Matematicheskogo Instituta Imeni V. A. Steklova},
    VOLUME = {216},
      YEAR = {1997},
     PAGES = {292--319},
      ISSN = {0371-9685,3034-1809},
   MRCLASS = {58F15 (22E40 58F10)},
  MRNUMBER = {1632177},
MRREVIEWER = {Raul\ Quiroga-Barranco},
}

@article {HPS,
    AUTHOR = {Hirsch, M. W. and Pugh, C. C. and Shub, M.},
     TITLE = {Invariant manifolds},
   JOURNAL = {Bull. Amer. Math. Soc.},
  FJOURNAL = {Bulletin of the American Mathematical Society},
    VOLUME = {76},
      YEAR = {1970},
     PAGES = {1015--1019},
      ISSN = {0002-9904},
   MRCLASS = {58F15 (57D30)},
  MRNUMBER = {292101},
MRREVIEWER = {A. Katok},
       DOI = {10.1090/S0002-9904-1970-12537-X},
       URL = {https://doi.org/10.1090/S0002-9904-1970-12537-X},
}

@article {LM,
    AUTHOR = {de la Llave, R. and Moriy\'on, R.},
     TITLE = {Invariants for smooth conjugacy of hyperbolic dynamical
              systems. {IV}},
   JOURNAL = {Comm. Math. Phys.},
  FJOURNAL = {Communications in Mathematical Physics},
    VOLUME = {116},
      YEAR = {1988},
    NUMBER = {2},
     PAGES = {185--192},
      ISSN = {0010-3616,1432-0916},
   MRCLASS = {58F15},
  MRNUMBER = {939045},
MRREVIEWER = {Georgi\u i\ Osipenko},
       URL = {http://projecteuclid.org/euclid.cmp/1104161299},
}

@article {GRH,
    AUTHOR = {Gogolev, A. and Rodriguez Hertz, F.},
     TITLE = {Smooth rigidity for codimension one {A}nosov flows},
   JOURNAL = {Proc. Amer. Math. Soc.},
  FJOURNAL = {Proceedings of the American Mathematical Society},
    VOLUME = {151},
      YEAR = {2023},
    NUMBER = {7},
     PAGES = {2975--2988},
      ISSN = {0002-9939,1088-6826},
   MRCLASS = {37D20},
  MRNUMBER = {4579371},
MRREVIEWER = {Italo\ Cipriano},
       DOI = {10.1090/proc/16177},
       URL = {https://doi.org/10.1090/proc/16177},
}

@article {GRH2,
    AUTHOR = {Gogolev, A. and Rodriguez Hertz, F.},
     TITLE = {Smooth rigidity for higher-dimensional contact {A}nosov flows},
      NOTE = {Reprint of Ukra\"in. Mat. Zh. {\bf 75} (2023), no. 9,
              1195--1203},
   JOURNAL = {Ukrainian Math. J.},
  FJOURNAL = {Ukrainian Mathematical Journal},
    VOLUME = {75},
      YEAR = {2024},
    NUMBER = {9},
     PAGES = {1361--1370},
      ISSN = {0041-5995,1573-9376},
   MRCLASS = {37E35 (37D20 37J55 53D25)},
  MRNUMBER = {4720721},
}

@article{GRH3,
    author = {Gogolev, A. and Rodriguez Hertz, F.},
    title = {Smooth rigidity for 3-dimensional volume preserving Anosov flows and weighted marked length spectrum rigidity},
    journal = {arXiv:2210.02295},
    year = {2022}
}

@article{GLRH,
    author = {Gogolev, A. and Leguil, M. and Rodriguez Hertz, F.},
    title = {Smooth rigidity for 3-dimensional dissipative Anosov flows},
    journal = {arXiv:2510.23872},
    year= {2025}
}

@article {Journe,
    AUTHOR = {Journ\'e, J.-L.},
     TITLE = {A regularity lemma for functions of several variables},
   JOURNAL = {Rev. Mat. Iberoamericana},
  FJOURNAL = {Revista Matem\'atica Iberoamericana},
    VOLUME = {4},
      YEAR = {1988},
    NUMBER = {2},
     PAGES = {187--193},
      ISSN = {0213-2230},
   MRCLASS = {58F15 (26B05 41A10 57R35 58F18)},
  MRNUMBER = {1028737},
MRREVIEWER = {Steven\ E.\ Hurder},
       DOI = {10.4171/RMI/69},
       URL = {https://doi.org/10.4171/RMI/69},
}

@article {RHJMD,
    AUTHOR = {Rodriguez Hertz, F.},
     TITLE = {Global rigidity of certain abelian actions by toral
              automorphisms},
   JOURNAL = {J. Mod. Dyn.},
  FJOURNAL = {Journal of Modern Dynamics},
    VOLUME = {1},
      YEAR = {2007},
    NUMBER = {3},
     PAGES = {425--442},
      ISSN = {1930-5311,1930-532X},
   MRCLASS = {37C85 (37C15 37D20 37D25)},
  MRNUMBER = {2318497},
MRREVIEWER = {Bachir\ Bekka},
       DOI = {10.3934/jmd.2007.1.425},
       URL = {https://doi.org/10.3934/jmd.2007.1.425},
}

@article {Cao,
    AUTHOR = {Cao, Y.},
     TITLE = {Non-zero {L}yapunov exponents and uniform hyperbolicity},
   JOURNAL = {Nonlinearity},
  FJOURNAL = {Nonlinearity},
    VOLUME = {16},
      YEAR = {2003},
    NUMBER = {4},
     PAGES = {1473--1479},
      ISSN = {0951-7715,1361-6544},
   MRCLASS = {37D20 (37A30)},
  MRNUMBER = {1986306},
MRREVIEWER = {Benoit\ Saussol},
       DOI = {10.1088/0951-7715/16/4/316},
       URL = {https://doi.org/10.1088/0951-7715/16/4/316},
}
\bibliographystyle{alpha}
\end{document}